\documentclass[12pt]{amsart}
\usepackage{amsmath,amsthm,amsfonts,amscd, amssymb, tikz, comment, mathrsfs, fullpage, graphicx, overpic, pdfsync, hyperref, appendix, csvsimple, subcaption, nicefrac}
\usepackage{xcolor}
\usepackage{faktor}
\usepackage{enumitem}
\usepackage{standalone}
\usepackage{outlines}
\title{The Tri-Pants Graph of the Twice-Punctured Torus}
\author{Katherine Betts}
\address{Katherine Betts: Furman University, Greenville, SC 29613}
\email{katherine.betts@furman.edu}
\author{Troy Larsen}
\address{Troy Larsen: Washington \& Lee University, Lexington, VA 24450}
\email{larsent22@mail.wlu.edu}
\author{Jeffrey Utley}
\address{Jeffrey Utley: University of Tennessee, Knoxville, TN 37996}
\email{jutley8@tennessee.edu}
\author{Avalon Vanis}
\address{Avalon Vanis: Bryn Mawr College, Bryn Mawr, PA 19010}
\email{avanis@brynmawr.edu}

\newcommand{\Z}{\mathbb{Z}}

\newcommand{\Q}{\mathbb{Q}}

\newcommand{\A}{\mathcal{A}}
\newcommand{\p}{\mathcal{P}}
\newcommand{\twicepuncturedtorus}{T^2\setminus\{\circ,\bullet\}}
\newcommand{\oncepuncturedtorus}{T^2\setminus\{\circ\}}
\newcommand{\tripantsgraph}{\mathcal{TP}}
\newcommand{\fareydual}{\mathcal{F}^*}

\newcommand{\homeoplus}{\text{Homeo}^+(\twicepuncturedtorus,\{\circ,\bullet\})}
\newcommand{\pmodtorus}{\text{PMod}(\twicepuncturedtorus)}
\newcommand{\oncepuncturedtorusblack}{T^2\setminus\{\bullet\}}
\newcommand{\Push}{\operatorname{Push}}
\newcommand{\gammahat}{\hat{\gamma}}
\newcommand{\alphahat}{\hat{\alpha}}
\newcommand{\betahat}{\hat{\beta}}

\newtheorem{thm}{Theorem}[section]

\newtheorem{lem}[thm]{Lemma}
\newtheorem{cor}[thm]{Corollary}
 
\newtheorem{theoremx}{Theorem}
\newtheorem{manualtheoreminner}{Theorem} 
\newenvironment{manualtheorem}[1]{%
    \renewcommand\themanualtheoreminner{#1}%
    \manualtheoreminner
}{\endmanualtheoreminner}
\theoremstyle{definition}
\newtheorem{defn}[thm]{Definition}

\theoremstyle{remark}
\newtheorem{remark}[thm]{Remark}
\newtheorem{claim}[thm]{Claim}
\numberwithin{equation}{section}
\let\oldtocsection=\tocsection
\let\oldtocsubsection=\tocsubsection
\let\oldtocsubsubsection=\tocsubsubsection
\renewcommand{\tocsection}[2]{\hspace{0em}\oldtocsection{#1}{#2}}
\renewcommand{\tocsubsection}[2]{\hspace{2em}\oldtocsubsection{#1}{#2}}
\renewcommand{\tocsubsubsection}[2]{\hspace{3em}\oldtocsubsubsection{#1}{#2}}
\begin{document}
\begin{abstract}
We investigate the structure of the tri-pants graph, a simplicial graph introduced by Maloni and Palesi, whose vertices correspond to particular collections of homotopy classes of simple closed curves of the twice-punctured torus, called tri-pants, and whose edges connect two vertices whenever the corresponding pants differ by suitable elementary moves. In particular, by examining the relationship between the tri-pants graph and the dual of the Farey complex, we prove that the tri-pants graph is connected and it has infinite diameter.
\end{abstract}
\maketitle
\tableofcontents

\section*{Introduction}

Let $T^2$ be a fixed orientable closed surface of genus $1$, and let $\bullet, \circ$ be two distinct marked points of $T^2$. We will call the \emph{twice-punctured torus} the punctured surface $T^2 \setminus \{\circ, \bullet\}$. The aim of this paper is to study the structure of the \emph{tri-pants graph}, a simplicial graph introduced by Maloni and Palesi and defined in terms of particular sets of essential simple closed curves on the twice punctured torus, called \emph{tri-pants}. A tri-pant can be described as a collection $T = \{\alpha, \alpha', \beta, \beta', \gamma, \gamma' \}$ of six distinct homotopy classes of non-separating, essential, simple closed curves on the twice-punctured torus that satisfy the following properties:
\begin{enumerate}
    \item every pair $\{\alpha, \alpha'\}$, $\{\beta, \beta'\}$, $\{\gamma, \gamma' \}$ describes a \emph{pants decomposition} of $T^2 \setminus \{\circ, \bullet\}$, i.e. it can be represented by disjoint simple closed curves that cut the surface into two once-punctured annuli;
    \item every pair of homotopy classes of curves belonging to distinct pants decompositions have minimal intersection number equal to $1$.
\end{enumerate}
We say that two tri-pants $T$ and $T'$ \emph{differ by an elementary move} if they share two out of three pants decompositions. The tri-pants graph $\mathcal{TP}$ can then be described as the simplicial graph that has a vertex for each tri-pant of the twice-punctured torus, and has an edge joining two vertices when the corresponding tri-pants differ by an elementary move.

The tri-pants graph can be considered as a higher complexity analogue of a similar simplicial graph defined in terms of systems of curves on the once-punctured torus, namely the \emph{dual of the Farey complex} $\mathcal{F}^*$, in light of the following observations: a pants decomposition of the once-punctured torus is given by (the homotopy class of) a single essential simple closed curve; each vertex of $\mathcal{F}^*$ can be interpreted as a triple of homotopy classes of essential simple closed curves intersecting exactly once, and therefore as the data of three pants decompositions with minimal intersection; two vertices of $\mathcal{F}^*$ are connected by an edge whenever the corresponding triples of homotopy classes differ by exactly one element. The first result that we will describe concerns the topology of the tri-pants graph. In particular, we will prove:

\begin{theoremx}\label{thm: tri-pants graph connected}
    The tri-pants graph is connected.
\end{theoremx}

\noindent Equivalently, every pair of tri-pants of the twice-punctured torus is connected by a finite sequence of elementary moves.

In fact, we can provide a more detailed description of the structure of the tri-pants graph in terms of a natural projection $\pi : \tripantsgraph \to \mathcal{F}^*$ onto the dual of the Farey complex. Let $\eta : \twicepuncturedtorus \to T^2 \setminus \{\circ\}$ denote the natural inclusion map. It turns out that, if $T = \{\alpha, \alpha', \beta, \beta', \gamma, \gamma' \}$ is a tri-pant, then the images of the curves in $T$ under $\eta$ provide a triple of essential simple closed curves of the once-punctured torus that intersect exactly once, up to homotopy. In particular, there is a natural projection from the set of vertices of $\tripantsgraph$ into $\mathcal{F}^*$. As we will describe in Section 3, adjacent vertices of the tri-pants graph are mapped under the projection onto vertices of $\mathcal{F}^*$ that are either equal, or adjacent. This allows us to construct a continuous surjective map $\pi : \tripantsgraph \to \mathcal{F}^*$ between the graphs, naturally associated to the inclusion $\eta : \twicepuncturedtorus \to T^2 \setminus \{\circ\}$. We will deduce Theorem \ref{thm: tri-pants graph connected} from the following result:

\begin{theoremx}\label{thm: fibers connected}
    For every vertex $v$ of $\mathcal{F}^*$, the fiber $\pi^{-1}(v)$ is a connected subgraph of $\mathcal{TP}$.
\end{theoremx}

Our proof of Theorem \ref{thm: fibers connected} relies on the properties of the \emph{pure mapping class class group} $\text{PMod}(\twicepuncturedtorus)$, the group of homeomorphisms of the torus that fix the set $\{\circ, \bullet\}$ pointwise, up to isotopy relative to the punctures. In particular, the main technical tool that we will deploy is the Birman's exact sequence \cite{birman1969} (see also \cite{MC}), which relates the pure mapping class group of the twice-punctured torus to the mapping class group of the once-punctured torus.

Theorem \ref{thm: fibers connected} allows us to deduce information also about the metric structure of the tri-pants graph. If $T$ and $T'$ are two vertices of $\mathcal{TP}$, we define their distance to be the minimum number of elementary moves required to go from $T$ to $T'$. With respect to such metric, we will show the following:

\begin{theoremx}\label{thm: infinitediameter}
    The tri-pants graph has infinite diameter.
\end{theoremx}

Both Theorems \ref{thm: tri-pants graph connected} and \ref{thm: infinitediameter} extend properties that are well known for the dual of the Farey complex to the case of the tri-pants graph. Additional analogies can be drawn between the two graphs, in particular in relation to the topological interpretation of their edges. 

A vertex of the dual of the Farey complex can be interpreted as the data of an \emph{ideal triangulation} of the once-punctured torus $T^2 \setminus \{ \circ \}$, namely a triangulation of $T^2$ with exactly one vertex at the puncture $\circ$. Two vertices of $\mathcal{F}^*$ are then connected by an edge if and only if the associated ideal triangulations differ by a diagonal exchange. A similar while more sophisticated phenomenon happens also for the tri-pants graph. The data of a tri-pant turns out to be equivalent to a system of three homotopy classes of essential simple arcs on the twice-punctured torus based at the puncture $\bullet$ and satisfying certain topological constraints. In the study of $\mathcal{TP}$, we will observe that two tri-pants $T$ and $T'$ differ by an elementary move if and only if their corresponding tri-arcs differ by suitable elementary flips. In this case, the elementary flips can be divided into two types in relation to the map $\pi$: flips between the tri-arcs associated to two adjacent tri-pants $T$ and $T'$ are called \emph{big flips} when $\pi(T) = \pi(T')$, while are called \emph{small flips} when $\pi(T)$ and $\pi(T')$ are distinct adjacent vertices in $\mathcal{F}^*$. Theorem \ref{thm: fibers connected} can then equivalently phrased as: every pair of tri-pants $T$ and $T'$ lying in the same fiber of $\pi$ have associated tri-arcs that differ by a finite sequence of big flips. 

As a major difference between the two graphs and their topological structure, we will show that $\mathcal{TP}$ has non-trivial simple closed loops, phenomenon that does not occur in the dual of the Farey complex, which is a infinite tri-valent tree. As a bi-product of our study, we will also provide an alternative characterization of the notion of tri-pant:

\begin{theoremx}\label{thm: tripants and maximal collections}
   A set of homotopy classes of non-separating essential simple closed curves on the twice-punctured torus that intersect at most once is maximal with respect to the inclusion if and only if it is a tri-pant.
\end{theoremx}

In particular, every set of homotopy classes of non-separating essential simple closed curves on the twice-punctured torus that intersect at most once has cardinality $\leq 6$, and it is equal to $6$ if and only if it is a tri-pant.

\vspace{1em}

An underlying motivation for the introduction of the tri-pants graph and the investigation of its properties comes from the connection between the notion of tri-pant and the character variety of $\mathbb{F}_3 \cong \pi_1(T^2 \setminus \{\circ, \bullet \})$, the free group in $3$ generators, into $\textrm{PSL}(2,\mathbb{C})$. In particular, the tri-pants graph constitutes a promising analogue of the decorated versions of the complex of curves deployed in the works of Bowditch, Maloni, Palesi, Tan, and Yang (see \cite{bowditch1998}, \cite{M-P}, \cite{maloni2015}, \cite{tan2008}, \cite{maloni2018typepreserving}) in the study of the action of the mapping class group of $S$ on the $\textrm{PSL}(2,\mathbb{C})$-character variety $\mathfrak{X}(\pi_1(S),\textrm{PSL}(2,\mathbb{C}))$, for $S$ equal to the once-punctured torus, the four-holed sphere, and the three-holed punctured plane.

\vspace{1em}

\textbf{Outline of the paper.} We rely on numerous definitions and structures to analyze the tri-pants graph $\tripantsgraph$. In Section~\ref{sec: prelims} we provide preliminary insights on closed curves, proper arcs, and their respective homotopies, outlining some of their properties and introducing a standard notation. 

Section~\ref{sec:pants} is dedicated to the study of pants decompositions of the twice-punctured torus, and the notion of tri-pant. Throughout our exposition, a pants decomposition will always be assumed to be composed of \emph{non-separating} simple closed curves. We then describe the notion of {\em tri-arcs} of the twice-punctured torus in Section~\ref{sec: tri-arc}, and we illustrate a one-to-one correspondence between tri-pants and tri-arcs in Lemma~\ref{lem: bijection}. Section \ref{sec: maximality} will then focus on the characterization of tri-pants as maximal families of non-separating simple closed curves that intersect at most once, as described by Theorem \ref{thm: tripants and maximal collections}.

The study of the tri-pants graph and the proofs of Theorems \ref{thm: fibers connected}, \ref{thm: tri-pants graph connected}, and \ref{thm: infinitediameter} are the central subjects of Section \ref{sec: graphs}. We first recall the main properties of the Farey complex and its dual graph $\mathcal{F}^*$ in Section \ref{subsec:farey}. Section 3 focuses on the investigation of elementary moves between tri-pants and the construction of the tri-pants graph $\mathcal{TP}$. By means of the bijective correspondence between tri-arcs and tri-pants described by Lemma~\ref{lem: bijection}, we characterize elementary moves in terms of small and big flips between tri-arcs (see Section \ref{subsubsec:elementary moves}). We show that every vertex of $\mathcal{TP}$ has degree 9 in Lemma \ref{lem: nineadjacent}, and introduce the projection map $\pi : \mathcal{TP} \to \mathcal{F}^*$ in Section \ref{sec: pi}. In the final part of the paper (Section \ref{sec: connect}) we recall the properties of the pure mapping class group of the twice-punctured torus and its associated Birman's exact sequence, leading us to the proof of the connectedness of the fibers of the map $\pi$ (Theorem \ref{thm: fibers connected}). We conclude our exposition by showing that the tri-pants graph is connected (Theorem \ref{thm: tri-pants graph connected}) and it has infinite diameter (Theorem \ref{thm: infinitediameter}).

\vspace{1em}


\textbf{Acknowledgements.} This work was conducted as part of the Topology REU at the University of Virginia in Summer 2021. We extend our sincerest gratitude to Filippo Mazzoli for suggesting and supervising this project and for helping us in writing this manuscript, and to Ross Akhmechet and Alec Traaseth for their mentorship. We would like to thank Sara Maloni and Fr\'ed\'eric Palesi for sharing their work on the topic and an outline for the first theorems. We are also grateful for the generous support of the National Science Foundation DMS-1839968 (NSF RTG), DMS-1848346 (NSF CAREER), and Thomas Koberda’s Alfred P. Sloan Foundation Fellowship.

\section{Preliminaries}\label{sec: prelims}
\subsection{Curves and arcs on the twice-punctured torus}
In this paper, we focus on how curves behave within punctured surfaces. We define a surface with punctures to be as follows.
\begin{defn} [Surface with Punctures]
Let $\overline{S}$ be a compact, connected, orientable surface, possibly with boundary, and let $P$ be a finite set of marked points in the interior of $\overline{S}$. Then $S = \overline{S}\setminus P$ is a {\em surface with punctures} and is often also denoted $(\overline{S},P)$.
\end{defn}
Throughout the remainder of this paper, we will use the term ``surface" to mean a compact, orientable, connected surface possibly with punctures and with (possibly empty) boundary. We define homotopy of closed curves in the following way, in order to specify how $\alpha$ and $\beta$ relate or differ in both surfaces: 
\begin{defn}[Homotopy of Closed Curves]\label{def: homotopy of closed curves}
Let $\alpha, \beta: S^1 \to S$ be two closed curves in a surface $S$. We say that $\alpha$ and $\beta$ are {\em homotopic as closed curves} if there exists a continuous map $F: [0,1]\times S^1\to S$ that satisfies the following: 
\begin{itemize}
    \item[(1)] For every $s\in S^1$, we have $F(0,s) = \alpha(s)$.
    \item[(2)] For every $s\in S^1$, we have $F(1,s) = \beta(s)$.
\end{itemize}
\end{defn}
When two curves $\alpha$ and $\beta$ are homotopic, we write $\alpha\simeq\beta$. Part of our motivation in studying punctured surfaces comes from understanding these homotopy classes. We now define some relevant properties of closed curves for the purposes of our paper.
\begin{defn}[Properties of Closed Curves] \label{def: closed curves}
Let $\alpha: S^1\to S$ be a closed curve in any surface $S$.
\begin{enumerate}[label=(\alph*)]
    \item $\alpha $ is {\em simple} if (the map) is injective. Items (b)-(d) assume that $\alpha$ is simple.
    \item $\alpha$ is {\em parallel to the boundary} if it is homotopic to one of the boundary components.
    \item A simple closed curve $\alpha$ is {\em essential} if it does not bound a disk or a punctured disk and is not parallel to the boundary.
    \item A simple closed curve $\alpha$ is {\em separating} if $S\setminus \alpha(S^1)$ is not connected and $\alpha$ is {\em non-separating} if $S\setminus\alpha(S^1)$ is connected.
\end{enumerate}
\end{defn}
\begin{defn}[Isotopy of Closed Curves]
Let $\alpha$, $\beta$ be simple closed curves such that $\alpha\simeq\beta$, and let $F:[0,1]\times S^1\to S$ be the continuous map such that $F(0,s) = \alpha(s), F(1,s)=\beta(s)$ for every $s\in S^1$. If $F(t,\cdot)$ is injective for all $t\in[0,1]$, then we say that $\alpha$ and $\beta$ are {\em isotopic as closed curves}.

\end{defn}
As outlined in the following section, we will be able to define a one-to-one correspondence between specific pairs of homotopy classes of essential, non-separating simple closed curves on the twice-punctured torus $\twicepuncturedtorus$ and homotopy classes of essential simple arcs on the once-punctured torus, based at the puncture. This will give us a better intuition of how these curves behave within our surface. 
1
 Thus, we define proper arcs within punctured surfaces as well as their properties and the notion of homotopy of arcs.

\begin{defn}[Proper Arcs]\label{def: proper arc}
Let $S$ be a surface, with set of punctures $P$. A {\em proper arc} in $S$ (or in $(\overline{S},P)$) is a continuous map $a: [0,1] \to \overline{S}$ such that $a^{-1}(P \cup \partial S) = \{0,1\}$.
\end{defn}

\begin{defn}[Homotopy of Arcs]\label{def: homotopy of arcs}
Let $S = (\overline{S}, P)$ be a surface with punctures $P$. Let $a,b: [0,1]\to \overline{S}$ be two proper arcs. We say that $a$ and $b$ are {\em homotopic as proper arcs} if there is a continuous map $F: [0,1]\times[0,1] \to \overline{S}$ that satisfies the following:
\begin{itemize}
    \item[(1)] $F(t,\cdot)$ is a proper arc for all $t\in[0,1]$.
    \item[(2)] $F(0,s) = a(s)$ for every $s \in [0,1]$.
    \item[(3)] $F(1,s) = b(s)$ for every $s \in [0,1]$.
\end{itemize}
We will continue to use the notation $a\simeq b$ if $a$ and $b$ are homotopic as proper arcs. Additionally, we say that $a$ and $b$ are {\em isotopic as proper arcs} if $F(t,\cdot)$ is injective for all $t\in[0,1]$.
\end{defn}

\begin{defn}[Properties of Arcs]\label{def: arcs} Let
 $a: [0,1]\to \overline{S}$ be a proper arc in a surface $S$. 
 \begin{itemize}
     \item[(a)] $a$ is called {\em simple} if $\alpha|_{(0,1)}$ is injective. Items (b)-(c) assume that $\alpha$ is simple.
     \item[(b)] A simple proper arc $a$ is {\em parallel to the boundary} if it is homotopic into a boundary component.
     \item[(c)] A simple proper arc $a$ is {\em essential} if it is not parallel to the boundary, not homotopic to a marked point of $S$, and does not bound a punctured disk.
     \item[(d)] If $b:[0,1]\to\overline{S}$ is also a simple proper arc and $F:[0,1]\times[0,1]\to\overline{S}$ is a homotopy from $a$ to $b$ such that $F(t,\cdot)$ is injective for all $t\in[0,1]$ then $a$ and $b$ are isotopic as proper arcs.
 \end{itemize}
\end{defn}
\subsection*{A note on notation}
For the sake of consistency and transparency throughout the remainder of this paper, we feel the need to outline a standard notation for curves and arcs.  When discussing curves, we make use of the Greek alphabet. We denote homotopy classes of simple closed curves with unaccented characters; $\alpha$, $\beta$, $\gamma$, for example. We indicate simple representatives of these classes by ``hatting" their characters. For instance, we may choose $\alphahat$ as a simple representative of the set $\alpha$. A similar pattern occurs when describing arcs, though we make use of the Latin alphabet. We denote homotopy classes of proper arcs by $a$, $b$, $c$, for example, and choose their respective simple representatives $\hat{a}$, $\hat{b}$, and $\hat{c}$. Additionally, although closed curves and proper arcs are defined to be continuous maps, when referencing either object we will identify the map with its image. That is, for a proper arc $\hat{a}:[0,1]\to\overline{S}$ and a closed curve $\alphahat:S^1\to\overline{S}$, we will let $a$ denote the image $\hat{a}[0,1]$ and $\alphahat$ denote the image $\alphahat(S^1)$.

\begin{defn}[Intersection Number]
Let $\alpha$ and $\beta$ be homotopy classes of simple closed curves in the torus $T^2$. For representatives $\hat{\alpha} \in\alpha,\hat{\beta}\in\beta$, define $n(\hat{\alpha},\hat{\beta})$ to be the number of times $\hat{\alpha}$ and $\hat{\beta}$ intersect. Then we define the {\em (geometric) intersection number} of $\alpha$ and $\beta$ by
$$i(\alpha,\beta):= \min\{n(\hat{\alpha},\hat{\beta})\:|\:\hat{\alpha} \in\alpha,\hat{\beta}\in\beta\}. $$
Thus, the intersection number is the minimum number of times that representatives of the two classes intersect. More results on intersection number can be found in \cite{MC}.
\end{defn}

Now that we have compiled a list of important properties of how closed curves and proper arcs behave within a punctured surface, we introduce the surface that will be the focal point of the remainder of our paper: the twice-punctured torus. 

\begin{defn}[Twice Punctured Torus]
Let $T^2$ be the torus, and let $P=\{\bullet, \circ\}$ be a set of two marked points in the interior of $T^2$. We define $S = T^2\setminus P$ to be the twice-punctured torus and denote it as $\twicepuncturedtorus$. 
\end{defn}

\begin{figure}[htbp]
    \centering
    \includegraphics[width=6cm]{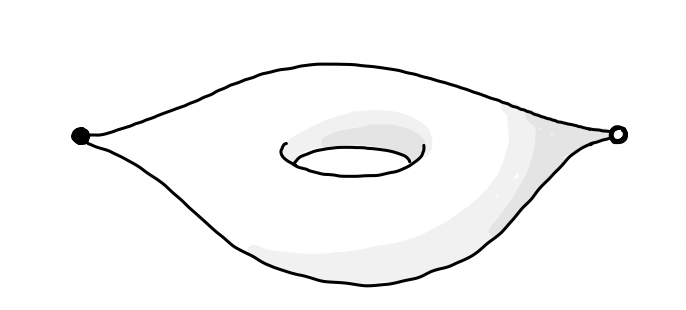}
    \caption{The twice-punctured torus.}
    \label{fig:twicepuncturedtorus}
\end{figure}
\section{(Tri-)Pants and Decompositions}\label{sec:pants}

\subsection{Pants decompositions} In this section, we investigate the relationship between the twice-punctured torus and collections of pairs of pants through the construction of pants decompositions. We will first present formal definitions of pants decompositions and then proceed to outline some of their important properties. 

\begin{defn}
Let $S_{g,b,n}$ denote the homeomorphism class of surfaces with genus $g$, number of boundary components $b$, and number of punctures $n$. 
\end{defn} 

\begin{defn}
A {\em pair of pants} is a surface $P$ in the homeomorphism class $S_{0,b,n}$, where $b+n=3$.
\end{defn}
In most extant literature, $P$ has three boundary components and no punctures. We extend this definition by allowing punctures to take the place of boundary components. In this less rigid construction, the sum of boundary components and punctures must equal three. Note that a pair of pants $P$ does not contain any non-separating simple closed curves, nor does it contain an essential simple closed curve \cite{MC}.

\begin{defn}[Pants Decomposition]\label{def: pantsdecompI}
A {\em pants decomposition} of a surface $S$ is a collection of disjoint simple closed curves $(\gamma_i)_i$ so that every connected component of $S\setminus\bigcup_i \gamma_i$ is homeomorphic to the interior of a pair of pants.
\end{defn}

Note that we say that a collection $(\alpha_i)_i$ of homotopy classes of simple closed curves give a pants decomposition of $S$ if there exists a set of representatives from each class $(\alphahat_i)_i$ which form a pants decomposition of the surface.

This definition is equivalent to saying that a pants decomposition is a maximal collection of disjoint essential simple closed curves in $S$ that are pairwise non-homotopic, as proven in Section 8.3.1 of \cite{MC}. Although a pants decomposition may include separating curves in other settings, we will focus only on pants decompositions composed of non-separating curves in this paper. It is important to note that a surface admits a pants decomposition if and only if $\chi(S) < 0$. This result follows from the fact that $\chi(S) = -n$ for any surface $S$ which may be decomposed into $n$ disjoint pairs of pants, as shown in Section 8.3.1 of \cite{MC}. 

\begin{lem}\label{lem: curvesinpantsdecomp}
Let $S$ be a surface so that $\chi(S)<0$ and with genus $g$, $b$ boundary components, and $n$ punctures. Any pants decomposition of $S$ divides the surface into $2g -2 + b + n$ pairs of pants and consists of $3g - 3 + b + n$ curves.
\end{lem}

A proof of Lemma~\ref{lem: curvesinpantsdecomp} (among many other results pertaining to pants decompositions) can be found in \cite{MC}. From now on, this paper will primarily focus on the twice-punctured torus. Lemma~\ref{lem: curvesinpantsdecomp} shows that a pants decomposition of $\twicepuncturedtorus$ consists of exactly two curves and splits $\twicepuncturedtorus$ into two pairs of pants. We will further restrict our study of pants decompositions to only those satisfying the following definition.

\begin{defn}\label{def: pantsdecompsoftpt}
Let $\p$ be the set of all pairs of distinct homotopy classes of non-separating, essential simple closed curves $\alpha, \alpha '$ in $\twicepuncturedtorus$ such that $\{\alpha, \alpha '\}$ determines a pants decomposition.
\end{defn}

\begin{lem}\label{lem: pantsdecomp}
Every pants decomposition $\{\alpha,\alpha'\}\in\p$ separates the twice-punctured torus into two surfaces, both homeomorphic to a once-punctured annulus.
\end{lem}

\begin{proof}
By Lemma~\ref{lem: curvesinpantsdecomp}, we may be sure that a pants decomposition of $\twicepuncturedtorus$ consists of two homotopy classes $\alpha,\alpha'$ and divides the surface into two pairs of pants, $P_1$ and $P_2$. For sake of contradiction, suppose that $P_1$ contains both $\bullet$ and $\circ$. Then $P_1$ must have exactly one boundary component as $b(P_1)+n(P_1)=3$. Since $b(\twicepuncturedtorus)=0$, this boundary component arises from cutting along a representative of $\alpha$ or $\alpha'$. Without loss of generality, assume that it arises from cutting along $\hat{\alpha}\in\alpha$. However, this implies that $\hat{\alpha}$ is separating, contradicting the fact that $\{\alpha,\alpha'\}\in\p$. Then $P_1$ and $P_2$ each must contain exactly one of the punctures, since applying the same argument to $P_2$ gives the same contradiction. Therefore $P_1$ and $P_2$ each have two boundary components and one puncture, and thus both are homeomorphic to a punctured annulus.
\end{proof}

\begin{figure}[htbp]
    \centering
    \includegraphics[width = 10 cm]{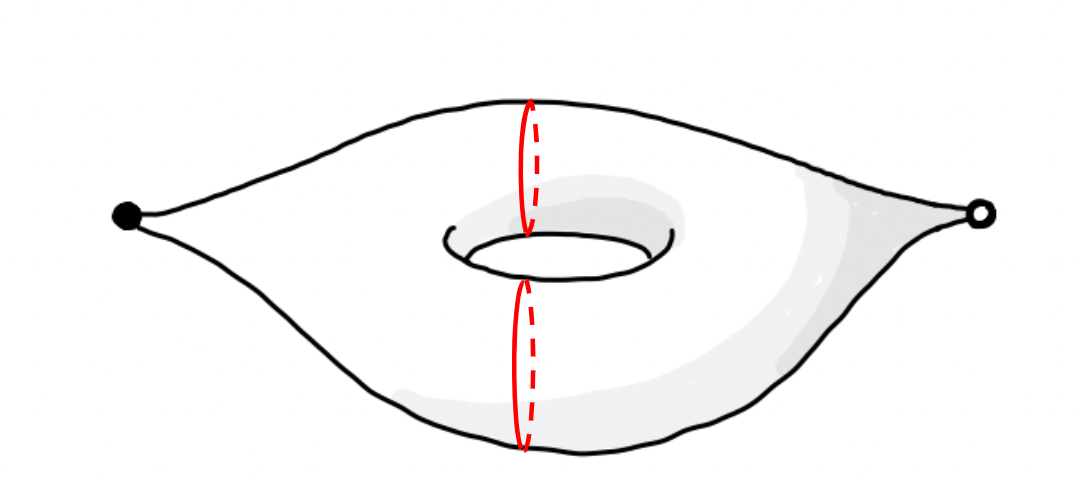}
    \caption{Two curves on $\twicepuncturedtorus$ that determine a pants decomposition.}
    \label{fig: meridiandecomp}
\end{figure}

\begin{cor}
Suppose that two curves $\hat{\alpha},\hat{\alpha}'$ determine a pants decomposition of $\twicepuncturedtorus$. Then $\hat{\alpha} \simeq \hat{\alpha}'$ when viewed as curves on $T^2$, but $\hat{\alpha} \not\simeq \hat{\alpha}'$ on the twice-punctured torus.
\end{cor}\label{cor: homotopicpantsdecomp}
 
This corollary will be crucial in our study of tri-pants and follows directly from Lemma~\ref{lem: pantsdecomp}. Further information about pants decompositions and most omitted proofs from this section can be found in Section 8.3.1 of \cite{MC}. For the following lemmas, we define $\eta: \twicepuncturedtorus\to T^2$ to be the inclusion map and $\eta_*$ the induced map the sends homotopy classes of essential simple closed curves in $\twicepuncturedtorus$ to those in $T^2$.

\begin{lem}\label{lem:intleq2}
Let $\alpha, \beta$ be homotopy classes in $\twicepuncturedtorus$ such that $i(\alpha, \beta) < 2$. Then $i(\alpha,\beta) = i(\eta_*(\alpha),\eta_*(\beta))$. 
\end{lem}
\begin{proof}
We approach this claim from the contrapositive. Let $\alpha, \beta$ be distinct homotopy classes in $\twicepuncturedtorus$ and let $\alphahat \in \alpha$ and $\hat{\beta} \in \beta$ be representatives of these classes that are in minimal position. Suppose that $i(\eta_*(\alpha),\eta_*(\beta)) \neq i(\alpha,\beta)$. Then $\eta\circ\alphahat$ and $\eta\circ\hat{\beta}$ must bound a bigon on $T^2$ in order to reduce intersection number under $\eta_*$ without changing either homotopy class. It follows that $\alphahat$ and $\betahat$ bound a punctured bigon on $\twicepuncturedtorus$, and $i(\alpha,\beta) \geq 2$. We thus conclude that if $i(\alpha,\beta) < 2$, then $i(\alpha,\beta) = i(\eta_*(\alpha),\eta_*(\beta))$.
\end{proof}

\begin{lem}\label{lem: cardleq2}
If $X$ is a set of homotopy classes of essential non-separating simple closed curves in $\twicepuncturedtorus$ that pairwise intersect at most once, then for any homotopy class of simple closed curves $\alpha \in T^2$, the cardinality of $\{\beta \in X | \eta_*(\beta) = \alpha \}$ does not exceed 2.
\end{lem}
\begin{proof}
First, let $\delta$ be a homotopy class of simple closed curves in $T^2$ and define the set $x=\eta_*^{-1}(\{\delta\})\cap X$. Proving our desired result is equivalent to proving that $card(x)\leq 2$. It follows from Definition~\ref{def: closed curves} that non-separating closed curves are essential in $\twicepuncturedtorus$ if and only if they are in $T^2$. Thus, if $\delta$ is not a homotopy class of essential curves, then $card(x)=0$. Therefore, we assume that $\delta$ is a class of essential curves. For sake of contradiction, assume $x= \{\alpha,\beta,\gamma\}$. Choose $\alphahat$, $\hat{\beta}$ and $\hat{\gamma}$ as their respective representatives so that they pairwise intersect minimally. Using the previous lemma, we claim that $i(\alpha, \beta) = i(\beta, \gamma) = i(\alpha, \gamma) = 0$. Indeed, $i(\alpha, \beta) = i(\eta_*(\alpha), \eta_*(\beta))=i(\delta, \delta)=0$. Thus, $\alphahat, \betahat,$ and $\gammahat$ are disjoint. Cutting $T^2$ along $\eta \circ \alphahat$ and $\eta \circ \betahat$ thus yields two annuli, which we will denote $A_1$ and $A_2$. Without loss of generality, assume that $\eta \circ \gammahat \subset A_1$. Since $\eta \circ \alphahat \simeq \eta \circ \hat{\beta} \simeq \eta \circ \hat{\gamma}$, we see that $\eta \circ \hat{\gamma}$ is parallel to the boundary of $A_1$. Then cutting along $\eta \circ \hat{\gamma}$ separates $A_1$ into two annuli, and we see that cutting along $\eta \circ \hat{\alpha}$, $\eta \circ \hat{\beta}$, and $\eta \circ \hat{\gamma}$ separates $T^2$ into three annuli. Note that if we remove exactly two points from this disjoint union of annuli, then at least one of the annuli will remain unchanged. Thus, we can conclude that at least two of these curves bound an annulus in $\twicepuncturedtorus$ and are therefore homotopic in this surface. Therefore, at least two of the homotopy classes in $x$ have to be identical, so $card(x) \leq 2$. 
\end{proof}
\subsection{Tri-Pants and their constructions}\label{sec: tri-pants}
As their names suggest, the study of tri-pants and the tri-pants graph are inherently related; specifically, one cannot expect to digest any results pertaining to the tri-pants graph without first comprehending tri-pants. In this section, we outline the criteria satisfied by tri-pants, describe some of their properties, and begin to discuss the relationships between differing tri-pants.
\begin{defn}\label{def: tripantsI}
A {\em tri-pant} on $\twicepuncturedtorus$ is a collection of six homotopy classes of simple closed curves $T=\{\alpha, \alpha', \beta, \beta', \gamma, \gamma'\}$ satisfying the following criteria: \begin{enumerate}
    \item[(1)] All curves in $T$ are essential and non-separating.
    \item[(2)] Without loss of generality, we have $\{\alpha, \alpha'\}$, $\{\beta, \beta'\}$, $\{\gamma, \gamma'\} \in \p$. 
    \item[(3)] For every pair of distinct elements $\zeta, \xi\in T$, we have $i(\zeta, \xi)=1$ unless $\{\zeta,\xi\}$ is equal to $\{\alpha, \alpha'\}$, $\{\beta, \beta'\}$, or $\{\gamma, \gamma'\}$.
\end{enumerate}
\end{defn}

\subsection*{Another note on notation}
We denote tri-pants by $T$ and $T'$ most frequently. Later, in Section~\ref{sec: tri-arc}, we define tri-arcs to be analogous sets of homotopy classes of proper arcs. To illustrate the bijection between tri-pants and tri-arcs (which we will investigate fully), we indicate tri-arcs with an asterisk in subscript; $T_*$ and $T_*'$ thus appear most frequently in our paper. Lastly, in Section~\ref{sec:moves}, we  introduce an oriented tri-arc which we label $T_{\pi_1},T_{\pi_1}'$ to emphasize that this object is a subset of the fundamental group of the surface.

Now that we have built more intuition on the properties of a tri-pant, we see that there are infinitely many choices of a tri-pant on $\twicepuncturedtorus$. At times, it can seem daunting to differentiate between them, so for the next portion of this paper we discuss similarity. 
\begin{defn}
Two tri-pants $T$ and $T'$ {\em differ by an elementary move} if they share two pants decompositions.
\end{defn}
While this definition gives us a concrete way to relate two tri-pants, it can be difficult to check if this property holds for a wide variety of reasons. If we view $T$ and $T'$ pictorially, the number (and complexity) of the curves often clouds our understanding. Alternatively, if we view $T$ and $T'$ algebraically as elements of the fundamental group of $\twicepuncturedtorus$, we quickly encounter the difficulties of working in a more complicated group structure. Thus, we seek simpler, but equally valid, representations of tri-pants to ease our study of their properties. In the next section, we outline one such representation and its relationship with tri-pants.
\subsection{The Tri-Arc presentation}\label{sec: tri-arc} Is there a way to represent tri-pants that simultaneously simplifies their pictorial and algebriac analysis? In this section, we answer this question with a resounding ``yes!" by introducing tri-arcs to our discussion. We will outline a one-to-one correspondence between pants decompositions in $\p$ and isotopy classes of simple, essential, unoriented arcs of the twice-punctured torus based at $\bullet$. In the process of studying tri-arcs, we reduce the number of pictorial components from six to three and reduce the complexity of working in the fundamental group by a factor of $\Z$.

\begin{defn}
Define $\A$ to be the set of isotopy classes of essential, unoriented simple arcs on $\twicepuncturedtorus$ based at $\bullet$ (that is, $a(0)=a(1)=\bullet$).
\end{defn}

\begin{defn}
We say that two homotopy (or isotopy) classes of arcs $a$ and $b$ based at the same point in a surface $S$ {\em bound a cylinder} if there exist representatives $\hat{a}$ and $\hat{b}$ so that a connected component of $S\setminus(\hat{a}\cup\hat{b})$ is homeomorphic to $S^1\times(0,1)$, the interior of a cylinder.
\end{defn}

\begin{defn}
A {\em tri-arc} $T_*=\{a, b, c\}$ is a collection of three distinct elements of $\A$ that pairwise intersect only at $\bullet$ and do not bound a cylinder in $\twicepuncturedtorus$.
\end{defn}

We now begin to investigate the relationship between tri-pants and tri-arcs. Define the map $\Phi: \A\to\p$ where $a\mapsto\{\alpha_a, \alpha'_a\}$ as follows: for $a\in\A$, choose a simple representative $\hat{a}$ and consider the simple closed curve $\hat{a}\cup\{\bullet\}\subset\oncepuncturedtorus$. Take an open tubular neighborhood $N_{\hat{a}}$ of $\hat{a}\cup\{\bullet\}$ homeomorphic to $S^1\times(0,1)$. Then $\partial N_{\hat{a}}\subset\twicepuncturedtorus$ consists of two simple closed curves $\alphahat_a,\alphahat'_a$, by the construction of $N_{\hat{a}}$. We define $\alpha_a,\alpha'_a$ to be the homotopy classes (as simple closed curves in $\twicepuncturedtorus$) of $\alphahat_a,\alphahat'_a$, respectively. 

\begin{claim}\label{clm: a, p corresp}
We claim that $\{\alpha_a, \alpha'_a\}$ is an element of $\mathcal{P}$, independent of the chosen neighborhood of $a$ (see e. g. \cite[Theorem 5.3]{hirsch} for the uniqueness of tubular neighborhoods up to isotopy), and the map $\Phi : \mathcal{A}\to \mathcal{P}$ where $a\mapsto \{\alpha_a, \alpha'_a\}$ is a bijection. A proof of this statement follows from the two subsequent lemmas.
\end{claim}

\begin{lem}\label{lem: alpha in P}
Let $a\in\A$. Then $\Phi(a)=\{\alpha_a, \alpha'_a\}$ is an element of $\p$. 
\end{lem}

\begin{proof}
We must first prove that there exist representatives $\alphahat_a,\alphahat'_a$ of $\alpha_a, \alpha_a'$, respectively, which are non-separating and non-peripheral. To do this, first fix a representative $\hat{a}\in a$. Note that there is only one homotopy class of non-trivial simple arcs in the punctured annulus $N_{\hat{a}}\setminus\{\bullet\}$, based at $\bullet$. We may take a simple representative $\hat{a}$ and then define the simple closed curve $\tilde{a}:=\hat{a} \cup \{\bullet\}\subset\oncepuncturedtorus$, which is homotopic to both representatives $\hat{\alpha}_a$ and $\hat{\alpha}_a'$ in $\oncepuncturedtorus$. Since $\hat{a}$ is essential as an arc in $\twicepuncturedtorus$, it does not bound a disk or a punctured disk within this surface, implying that $\tilde{a}$ is also essential as a simple closed curve in $\oncepuncturedtorus$ and therefore $\hat{\alpha}_a, \hat{\alpha}_a'$ are as well. Note that all simple separating curves on $T^2$ bound disks, so all simple separating curves on $\oncepuncturedtorus$ are non-essential. Thus, $\alphahat_a$ and $\alphahat'_a$ are non-separating in $\oncepuncturedtorus$, so $(\oncepuncturedtorus)\setminus \hat{\alpha}_a$ is connected. Because connectedness of a surface is not changed by removing a single point, $(\twicepuncturedtorus)\setminus\hat{\alpha}_a$ is also connected. The same argument applies to $\hat{\alpha}_a'$, so both representatives are non-separating in $\twicepuncturedtorus$. Now, note that if a simple curve is peripheral in $\twicepuncturedtorus$, then it either bounds a disk or a punctured disk in $\oncepuncturedtorus$. However, because $\alphahat_a$ and $\alphahat'_a$ are essential in $\oncepuncturedtorus$, this cannot be the case. Therefore, $\alphahat_a$ and $\alphahat_a'$ are also non-peripheral in $\twicepuncturedtorus$.

Now, we prove that $\alphahat_a,\alphahat_a'$ gives a pants decomposition of $S$. Because $N_{\hat{a}}\setminus\{\bullet\}$ is homeomorphic to the interior of a punctured annulus, if we let $\overline{N_{\hat{a}}}$ denote $N_{\hat{a}}\cup\partial N_{\hat{a}}$, then $\chi(\overline{N_{\hat{a}}}\setminus \{\bullet\})=-1$. Let $M:=\twicepuncturedtorus\setminus N_{\hat{a}}$. Then $$\chi(\overline{N_{\hat{a}}}) + \chi(M) = \chi(\twicepuncturedtorus)=-2,$$ so $\chi(M) = -1$. We know that $M$ must have two boundary components, $\alphahat_a$ and $\alphahat_a'$, as well as one puncture, $\circ$, implying that $M$ has genus 0 and is thus also homeomorphic to a punctured annulus. Because punctured annuli are pairs of pants, $\{\alphahat_a, \alpha_a'\}$ gives a pants decomposition of $\twicepuncturedtorus$. Therefore, $\{\alpha_a,\alpha_a'\}\in\mathcal{P}$.
\end{proof}

\begin{lem}\label{lem: bijection}
The map $\Phi: \A \to \p$ where $a\mapsto \{\alpha_a, \alpha'_a\}$ is a bijection.
\end{lem}
\begin{proof}
To prove bijectivity, we introduce a candidate inverse $\Psi: \p \to \A$. Let $\{\alpha, \alpha'\}\in\p$. By definition, cutting along disjoint representatives $\hat{\alpha},\hat{\alpha}'$ of these two homotopy classes yields two disjoint pairs of pants denoted $P_1$ and $P_2$, each containing a single puncture. Without loss of generality, assume that $\bullet\in P_1$ and consider $P_1$ as a once-punctured annulus. Inside this annulus, there exists a unique homotopy class of simple and non-trivial unoriented arcs based at $\bullet$. Let $a$ be this class of arcs and note that we can find a representative $\hat{a}\in a$ which is parallel to the boundary. That is, $\hat{a} \cup \{\bullet\}$ is homotopic as a simple closed curve to $\hat{\alpha}$ in $\oncepuncturedtorus$, so it is also essential as a simple closed curve in $\oncepuncturedtorus$. It follows that $\hat{a}$ is essential as an arc, as it does not bound a disk or a punctured disk. Therefore, we may conclude that $a\in\A$ and define $\Psi(\{\alpha,\alpha'\}):=a$. We can now show that $(\Phi\circ\Psi)(\{\alpha, \alpha'\})=\{\alpha, \alpha'\}$ and $(\Psi\circ\Phi)(a)=a$ for any $\{\alpha,\alpha'\}\in\p$ and $a \in \A$. First, consider $(\Phi\circ\Psi)(\{\alpha, \alpha'\})$. We denote $a:=\Psi(\{\alpha, \alpha'\})$. Taking any tubular neighborhood $N_{\hat{a}}$ of a representative $\hat{a}$ in $\oncepuncturedtorus$ we obtain an element $\{\beta,\beta'\} \in \p$ (by Lemma \ref{lem: alpha in P}) such that $\{\beta, \beta'\} = \Phi(a)$. However, as shown above, $\hat{a}\cup\bullet\simeq\hat{\alpha},\hat{\alpha}'$ as simple closed curves in $\oncepuncturedtorus$ since $\hat{\alpha},\hat{\alpha}'$ bound a tubular neighborhood of $\hat{a}$. Because the same is true for representatives $\hat{\beta},\hat{\beta}'$, these curves are also homotopic to $\hat{a}\cup\bullet$ in $\oncepuncturedtorus$. Because $\circ$ is not in either neighborhood by construction, this implies that $\hat{\alpha}\simeq\hat{\beta}$ and $\hat{\alpha}'\simeq\hat{\beta}'$ in $\twicepuncturedtorus$. Thus, $\{\alpha, \alpha'\} = \{\beta, \beta'\}$ so $(\Psi \circ \Phi)(\{\alpha,\alpha'\}) = \{\alpha, \alpha'\}$. Now consider $(\Psi \circ \Phi)(a)$. Taking a tubular neighborhood $N_{\hat{a}}$ of a representative $\hat{a}$ in $T^2\setminus\{\circ\}$ gives boundary curves $\partial N_{\hat{a}} = \{\alpha,\alpha'\} = \Phi(a) \in \p$ by Lemma~\ref{lem: alpha in P}. By construction, representatives of $\alpha, \alpha'$ bound a punctured annulus containing $\hat{a}$. Going through the process outlined for $\Psi$, we locate the unique homotopy class of non-trivial arcs based at the puncture $\bullet$. By uniqueness, this class must be $a$. Thus $(\Psi \circ \Phi)(a) = a$. It follows now that $\Psi = \Phi^{-1}$, and thus $\Psi$ must be bijective.
\end{proof}

Therefore, Claim~\ref{clm: a, p corresp} holds. However, the existence of a bijection between $\A$ and $\p$ is not enough to give us our desired one-to-one correspondence between tri-arcs and tri-pants. In order to prove this correspondence, we must prove that minimal intersection of the arcs $a$, $b$, and $c$ corresponds to minimal intersection of the associated pants decompositions $\Phi(a)$, $\Phi(b)$, and $\Phi(c)$.

\begin{claim}\label{clm: triarcs, tripants corresp}
There is a one-to-one correspondence between the sets of tri-arcs and tri-pants.
\end{claim}

The following lemmas are the key ingredients for the proof of this claim.

\begin{lem}\label{lem: abintersectonce}
Let $a$ and $b$ be two distinct elements of $\mathcal{A}$. There exist representatives $\hat{a}$ and $\hat{b}$ based at the puncture $\bullet$ that do not bound a cylinder inside $T^2 \setminus \{\circ\}$ if and only if the associated pants decompositions $\Phi(a) = \{ \alpha, \alpha' \}$ and $\Phi(b) = \{ \beta, \beta' \}$ satisfy $i(\alpha, \beta) = i(\alpha, \beta') = i(\alpha', \beta) = i(\alpha', \beta') = 1$.
\end{lem}

\begin{proof} 

($\Longrightarrow$) Assume first that there are representatives $\hat{a}\in a,\hat{b}\in b$ intersecting only at $\bullet$ that do not bound a cylinder. We now consider $\hat{a} \cup \{\bullet\}$ as a simple closed curve inside $T^2$, as opposed to a proper arc inside $\oncepuncturedtorus$. Cut $T^2$ along $\hat{a} \cup \{\bullet\}$ to yield an annulus. More formally, there exists a compact annulus $A$ with $\partial A :=\{\tilde{a},\tilde{a}'\}$ and a homeomorphism $h:\tilde{a}\to\tilde{a}'$ such that the quotient $\faktor{A}{x= h(x)}\cong T^2$, $\partial A$ projects onto $\hat{a}\cup\{\bullet\}$, and there exist two points $x\in \tilde{a},x'\in \tilde{a}'$ which project onto $\{\bullet\}$. Further, since $\hat{b}$ intersects $a$ only at $\bullet$, there is a unique lift $\tilde{b}\subset \text{int(A)}$ homeomorphic to $\tilde{b}$, where int$(A):=A\setminus\partial A$. In this construction, the puncture $\circ$ will correspond to some point $y$ that lies in the interior of $A$. We know that $\hat{b}$ must have endpoints contained in $\{x,x'\}$; however, we want to exclude the possibility that the endpoints of the path $\tilde{b}$ coincide. By way of contradiction, suppose that $\tilde{b}$ is based at $x$. We may then view $\tilde{b}$ as an element of $\pi_1(A,x)$. Suppose now that $\tilde{b}$ is trivial, which implies the existence of an open disk $B\subset \text{int}(A)$ with boundary $\tilde{b}\cup\{x\}\subset A$. $B$ then projects onto a disk in $\oncepuncturedtorusblack$ bounded by $\hat{b}$. Thus, $\hat{b}$ either bounds a disk or a punctured disk in $\twicepuncturedtorus$, both of which contradict the fact that $b\in\A$. Therefore, $\tilde{b}$ must not be trivial. There is only one non-trivial element of $\pi_1(A,x)$ (up to orientation) which has a simple representative. This choice is parallel to the boundary curve $\tilde{a}$ as in Figure \ref{fig: non-trivial b hat}. However, since $x\in\tilde{a}$, we see that $\tilde{a}\simeq\tilde{b}$ as arcs based at $x$ in $A$. It follows that $\hat{a} \simeq \hat{b}$ as arcs in $T^2$ based at $\bullet$. Since we assumed $a\neq b$ as arcs based at $\bullet$ in $\oncepuncturedtorus$, this implies that $\circ$ lies in the region homeomorphic to the interior of an annulus bounded by $\hat{a}$ and $\hat{b}$. Because $\circ$ is disjoint from $a$, there is a unique point $y\in A$ disjoint from $\tilde{b}$ that projects to $\circ$. This implies that $y$ lies in the lift of the aforementioned annulus, the region in $A$ bounded by $\tilde{b}$ and $\tilde{a}$ in Figure~\ref{fig: non-trivial b hat}.

\begin{figure}[htbp]
    \centering
    \begin{overpic}
    [scale=0.6]{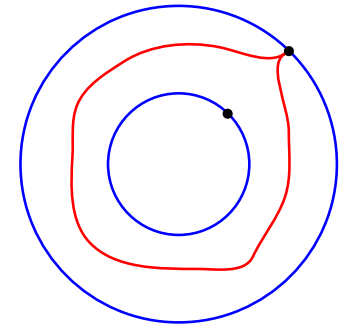}
    \end{overpic}
    \put(-31,134){$x$}
    \put(-57,103.5){$x'$}
    \put(-4,80){$\Dot{a}$}
    \put(-105.5,75.5){$\Dot{a}'$}
    \put(-80,14.5){$\Dot{b}$}
    \caption{A non-trivial arc $\Dot{b}\in\pi_1(A,x)$.}
    \label{fig: non-trivial b hat}
\end{figure}
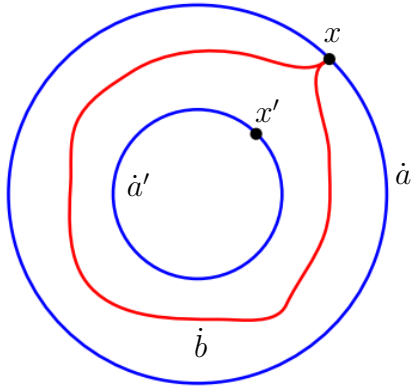

However, $\tilde{b}$ and $\tilde{a}'$ then bound a cylinder in $A\setminus\{y\}$. That is, there is a component $A'$ of int$(A)\setminus\tilde{b}$ which is homeomorphic to the interior of an annulus $S^1\times(0,1)$. $A'$ then projects to a connected component of $\oncepuncturedtorusblack\setminus(\hat{a}\cup\hat{b})$ homeomorphic to $S^1\times(0,1)$. Since this region is contained in int$(A)$, it projects onto a region bounded by $\hat{a}$ and $\hat{b}$ which is homeomorphic to the interior of an annulus, contradicting the hypothesis. Thus, $\tilde{b}$ cannot lay within $\pi_1(A,x)$, so it is not an arc based at just one point. That is, $\tilde{b}$ must connect $x$ and $x'$. Since $\{x,x'\}$ projects onto $\bullet$, this implies that the intersection of $\hat{a}$ and $\hat{b}$ in $\twicepuncturedtorus$ is transverse and thus resembles Figure~\ref{fig: transverse} in a sufficiently small neighborhood of $\bullet$. Recall that the construction of $\Phi(a)$ does not depend on the choice of tubular neighborhood $N_{\hat{a}}$ and we may then choose the neighborhoods $N_{\hat{a}}$ and $N_{\hat{b}}$ to be as small as we would like. In fact, we may find a chart $(U,\phi_U)$ consisting of an open neighborhood $U\subset T^2$ of $\bullet$ and a homeomorphism $\phi_U:U\to \phi_U(U)\subset\mathbb{R}^2$ and then choose $N_{\hat{a}}$ and $N_{\hat{b}}$ such that $\overline{N_{\hat{a}}}\cap\overline{N_{\hat{b}}}\subset U$, since $\hat{a}$ and $\hat{b}$ only intersect at $\bullet$. Then, $\phi_U(\overline{N_{\hat{a}}}\cap\overline{N_{\hat{b}}})$ is a closed rectangle in $\mathbb{R}^2$ resembling Figure~\ref{fig: simple intersection}. The corners of the resulting rectangle each give us our desired  intersections between $\alpha,\alpha',\beta,\beta'$. Since $\overline{N_{\hat{a}}}\cap\overline{N_{\hat{b}}}$ is contained fully within $U$, there are no other intersections. With this construction, we can see that our desired intersection number is always upheld.

\begin{figure}[htbp]
    \centering
    \includegraphics[scale=0.3]{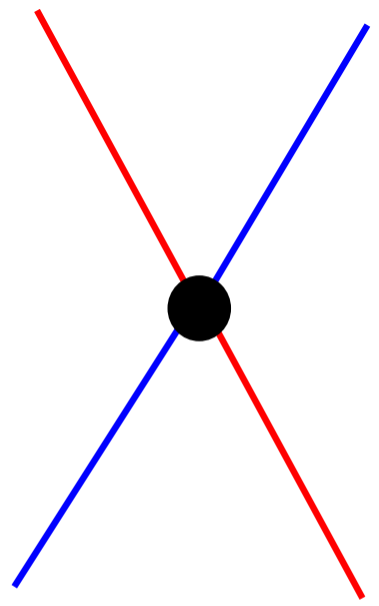}
    \caption{The local intersection of $\hat{a}$ and $\hat{b}$.}
    \label{fig: transverse}
\end{figure}

\begin{figure}[htbp]
    \centering
    \begin{overpic}
    [scale=0.7]{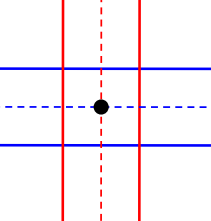}
    \end{overpic}
    \put(-123,78){$\alpha$}
    \put(-122,57.5){$\hat{a}$}
    \put(-123,38){$\alpha'$}
    \put(-81,-12.5){$\beta$}
    \put(-60,-13){$\hat{b}$}
    \put(-41,-12.5){$\beta'$}
    \caption{The set $\phi_U(\overline{N_{\hat{a}}}\cap\overline{N_{\hat{b}}})\subset\mathbb{R}^2$.}
    \label{fig: simple intersection}
\end{figure}
    
($\Longleftarrow$) Now, assume that $i(\alpha,\beta)=i(\alpha',\beta)=i(\alpha',\beta')=i(\alpha,\beta')=1$. Choose representatives $\alphahat,\alphahat',\betahat,\betahat'$ which intersect minimally. Then there exist punctured annuli $A_{\alpha}:=N_{\alpha}\setminus\{\bullet\}$ and $ A_{\beta}:=N_{\beta}\setminus\{\bullet\}$ bounded by $\alphahat,\alphahat'$ and $\betahat,\betahat'$ respectively, and such that $A_{\alpha}\cap A_{\beta}$ is a punctured square. Because $N_{\alpha},N_{\beta}$ are unpunctured annuli, within both sets there exists a unique homotopy class of non-trivial simple arcs based at $\bullet$. Let $a$ be this class in $N_{\alpha}$ and $b$ the class in $N_{\beta}$. Note that $a$ and $b$ are also non-peripheral, since $a$ is equivalent to $\alpha$ and $b$ is equivalent to $\beta$ as homotopy classes of simple closed curves in $\oncepuncturedtorus$. Further, note that since $A_{\alpha}\cap A_{\beta}$ is homeomorphic to a punctured square, we can choose representatives $\hat{a}\in a$ and $\hat{b}\in b$ which do not intersect in this set. By our construction, $\hat{a}$ and $\hat{b}$ only intersect at $\bullet$ in $\oncepuncturedtorus$. Consider now the annulus $A$ obtained by cutting along $\hat{a}$ and the homeomorphism $h$ defined as in the forward implication. By construction, $\hat{a}$ and $\hat{b}$ meet transversally as in Figure~\ref{fig: transverse}. Thus, $\hat{a}$ and $\hat{b}$ meet at $\bullet$ from different sides, implying that the endpoints of $\tilde{b}$ are distinct inside $A$. Note that cutting along a simple arc in a cylinder which connects the two boundary components results in a square. Thus, int$(A)\setminus\tilde{b}$ is homeomorphic to the interior of a square and will project onto $\oncepuncturedtorusblack\setminus(\hat{a}\cup\hat{b})$. This implies that $\twicepuncturedtorus\setminus(\hat{a}\cup\hat{b})$ is a connected space which is homeomorphic to a punctured square and not $S^1\times(0,1)$. Therefore, $\hat{a}$ and $\hat{b}$ do not bound a cylinder.
\end{proof}

\begin{lem} \label{lem: abcformtrip}
Let $a$, $b$ and $c$ be three distinct elements of $\mathcal{A}$. There exist representatives $\hat{a}$, $\hat{b}$ and $\hat{c}$ intersecting only at the puncture $\bullet$ that pairwise do not bound a cylinder inside $T^2 \setminus \{\circ\}$ if and only if the associated pants decompositions $\Phi(a) = \{ \alpha, \alpha' \}$, $\Phi(b) = \{ \beta, \beta' \}$ and $\Phi(c) = \{ \gamma, \gamma' \}$ provide a tri-pants of $\twicepuncturedtorus$.
\end{lem}

\begin{proof}
($\Longrightarrow$) Suppose that there exist representatives $\hat{a}, \hat{b}$, and $\hat{c}$ intersecting only at $\bullet$ that pairwise do not bound a cylinder inside $\oncepuncturedtorus$. By Lemma~\ref{lem: alpha in P}, the three pairs $\{\alpha,\alpha'\}$, $\{\beta,\beta'\}$, and $\{\gamma,\gamma'\}$ are all homotopy classes of essential and non-separating simple closed curves which determine pants decompositions of $\twicepuncturedtorus$. By the construction of the tubular neighborhood $N_{\hat{a}}$, there exist representatives of $\alpha,\alpha'$ which do not intersect. This result holds  for $\beta,\beta'$ and $\gamma,\gamma'$ as well. Applying Lemma~\ref{lem: abintersectonce} to each of the pairs $\{\hat{a},\hat{b}\}$, $\{\hat{a},\hat{c}\}$, and $\{\hat{b},\hat{c}\}$, we see that the associated pairs in $\{\alpha,\alpha',\beta,\beta',\gamma,\gamma'\}$ intersect pairwise once unless they determine a pants decomposition. Therefore, the set is a tri-pant.

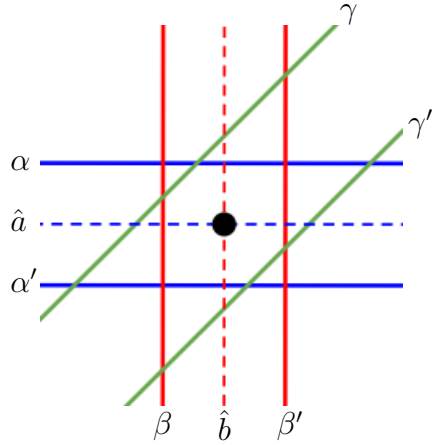
\begin{figure}[htbp]
    \centering
    \begin{overpic}
    [scale=0.8]{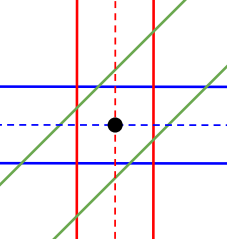}
    \put(-8,62){$\alpha$}
    \put(-8,46){$\hat{a}$}
    \put(-8,30){$\alpha'$}
    \put(29.5,-7.5){$\beta$}
    \put(46.5,-9){$\hat{b}$}
    \put(62,-7.5){$\beta'$}
    \put(78.5,102){$\gamma$}
    \put(96.5,72){$\gamma'$}
    \end{overpic}
    \vspace{0.3cm}
    \caption{The local depiction of $\phi_U(\hat{T}\cap U)$, $\phi_U(\hat{a}),$ and $\phi_U(\hat{b})$.}
    \label{fig: intersection model}
\end{figure}

($\Longleftarrow$) Assume that $T=\{\alpha,\alpha',\beta,\beta',\gamma,\gamma'\}$ is a tri-pant and let $\hat{T}=\{\alphahat, \hat{\alpha}', \betahat$, $\hat{\beta}', \gammahat,\hat{\gamma}'\}$ denote its representatives. Taking small enough tubular neighborhoods $N_{\hat{a}},N_{\hat{b}},N_{\hat{c}}$, we may find a chart $(U,\phi_U)$ containing an open neighborhood $U\subset\twicepuncturedtorus$ of $\bullet$ such that
$$\overline{N_{\hat{a}}}\cap\overline{N_{\hat{b}}},\overline{N_{\hat{a}}}\cap\overline{N_{\hat{c}}},\overline{N_{\hat{b}}}\cap\overline{N_{\hat{c}}}\subset U$$
and a homeomorphism $\phi_U:U\to\phi_U(U)\subset\mathbb{R}^2$. Since $a,b\in \A$, we can apply Lemma~\ref{lem: abintersectonce} to find representatives $\hat{a}$ and $\hat{b}$ which intersect only at $\bullet$ and do not bound a cylinder. It follows that a (local) picture of $\phi_U(\hat{T}\cap U)$, $\phi_U(\hat{a}),$ and $\phi_U(\hat{b})$ can be described as in Figure~\ref{fig: intersection model}, up to homotopy of $\hat{a}$ and $\hat{b}$. Notice that  $N_{\hat{a}}\cap N_{\hat{b}}$ forms a square with one diagonal parallel to $\gamma$ and $\gamma'$. Up to homotopy, we can choose this diagonal in such a way that it intersects $\hat{a}$ and $\hat{b}$ only at the puncture $\bullet$ within $N_{\hat{a}}\cap N_{\hat{b}}$. Because the diagonal is parallel to $\gamma,\gamma'$, we can extend it into an arc based at $\bullet$, contained in $N_{\hat{c}}$, and parallel to $\partial N_{\hat{c}}$. Since $N_{\hat{c}}$ is an annulus, this will be in the unique homotopy class of arcs in $N_{\hat{c}}$ that are non-trivial. By construction of the map $\Phi$, this class is $c$ and we will therefore denote our chosen arc as $\hat{c}$. Thus, we have found representatives $\hat{a}, \hat{b}, \hat{c}$ which intersect only at $\bullet$. By construction, notice that $\hat{a}, \hat{b}, \hat{c}$ all intersect transversely. Therefore we may appeal to the proof of the backwards direction of Lemma~\ref{lem: abintersectonce} in order to conclude that these representatives must not pairwise bound a cylinder.
\end{proof}

Lemma~\ref{lem: abcformtrip} then tells us that a triple $\{a,b,c\}\subset\A$ forms a tri-arc if and only if the pairs of curves in $\Phi(a),\Phi(b),\Phi(c)\in\p$ determine a tri-pant. Since $\Phi$ is bijective, this applies to all possible triples in $\A$ and $\p$. Therefore, we have shown an explicit bijective correspondence between tri-pants and tri-arcs. We will now describe some properties of tri-arcs in order to build topological intuition about these objects before discussing the tri-pants graph.

\begin{remark} \label{rmk:punctured_square}
If $a, b \in \A$ are as in the statement of Lemma~\ref{lem: abintersectonce}, then the complementary region of $\hat{a} \cup \hat{b}$ is a punctured square.
\end{remark}

This follows from the proof of Lemma~\ref{lem: abintersectonce}. As discussed, cutting along $\hat{a}$ gives an annulus with $\hat{b}$ corresponding to an arc in this space which connects the two boundary components. Cutting along this arc then gives a punctured square, with both pairs of opposite edges identified. One pair will project onto $\hat{a}$ and the other onto $\hat{b}$, as in Figure~\ref{fig: punctured square}. Now, we may look at two more remarks that follow from Remark~\ref{rmk:punctured_square}.

\begin{figure}[htbp]
    \centering
    \begin{overpic}
    [scale=0.5]{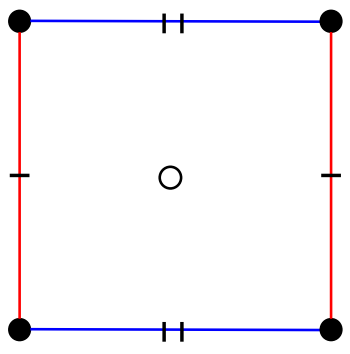}
    \end{overpic}
    \put(-68,-6.75){$\hat{a}$}
    \put(3,63.5){$\hat{b}$}
    \vspace{0.2cm}
    \caption{The complimentary region of $\hat{a}\cup\hat{b}$.}
    \label{fig: punctured square}
\end{figure}

\begin{remark}\label{rmk: fund grp generators}
We can show that, for any pair $a,b\in\A$ that pairwise intersect only at $\bullet$ and do not bound a cylinder, $\pi_1(\oncepuncturedtorus, \bullet)\cong\Z\cdot a\ast\Z\cdot b$.
\end{remark}

If we take such homotopy classes $a,b$, then Remark~\ref{rmk:punctured_square} tells us that cutting along these curves gives a punctured square. The once-punctured torus is homotopically equivalent to $\hat{a}\cup\hat{b\cup\{\bullet\}}$ and has the fundamental group $\Z\ast\Z$.

\begin{remark} \label{rmk:triplearcs}
If $a,b, c \in \A$ are as in Lemma \ref{lem: abcformtrip}, then the complementary regions of $\hat{a}\cup \hat{b}\cup \hat{c}$ are a triangle and a punctured triangle.
\end{remark}

This remark follows from Remark~\ref{rmk:punctured_square} and Lemma~\ref{lem: abcformtrip}. Cutting along representatives $\hat{a}$ and $\hat{b}$ gives us a punctured square where $\hat{c}$ will intersect $\hat{a}$ and $\hat{b}$ only at $\bullet$. That is, $\hat{c}$ does not intersect any of the edges, but is based in the four vertices. If $\hat{c}$ connected a single vertex to itself, then it would either bound a disk or a punctured disk, contradicting the fact that it is essential. If $\hat{c}$ connected two adjacent vertices, then it would be either homotopic to or bound a cylinder with $\hat{a}$ or $\hat{b}$, creating another contradiction. Therefore, $\hat{c}$ must connect two diagonal vertices, similar to Figure \ref{fig: tri-arc pic}. Thus, cutting along $\hat{c}$ gives a triangle and a punctured triangle.

\begin{figure}[htbp]
    \centering
    \begin{overpic}
    [scale=0.5]{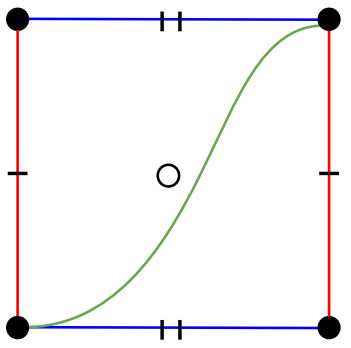}
    \end{overpic}
    \put(-70,-7){$\hat{a}$}
    \put(0,62){$\hat{b}$}
    \put(-57,54){$\hat{c}$}
    \vspace{0.2cm}
    \caption{The complimentary region of $\hat{a}\cup\hat{b}\cup\hat{c}$.}
    \label{fig: tri-arc pic}
\end{figure}

The question of how to determine whether two tri-arcs differ by an elementary move still looms large. Luckily, the tri-arc representation yields a simpler algebraic framework. By selecting an arbitrary orientation for each of the homotopy classes within the tri-arc, we can consider the arcs of a given tri-arc as elements of $\pi_1(\oncepuncturedtorus,\bullet)\cong\Z\cdot a \ast \Z\cdot b$, where $a$ and $b$ denote arbitrary generators. Note that Remark~\ref{rmk:punctured_square} holds regardless of the choice of orientation. Thus, every tri-arc therefore corresponds to a triple of ``words" written with the letters $a, a^{-1}, b,$ and $b^{-1}$. Lemma~\ref{lem: tri-arc decomp} follows as an immediate consequence of Remark~\ref{rmk:triplearcs}.
\begin{lem}\label{lem: tri-arc decomp}
Let $T_*=\{a, b, c\}$ be a tri-arc. If we consider $a, b, c$ as oriented elements of the fundamental group $\pi_1(\oncepuncturedtorus,\bullet)$, then $c = a^{\pm 1}b^{\pm 1}$ or $c = b^{\pm 1}a^{\pm 1}$. 
\end{lem}

Even still, the triples describing tri-arcs can easily yield unbroachable complexity. We seek an even simpler representation of tri-pants in order to understand the tri-pants graph.
\subsection{Tri-Pants and Maximality}\label{sec: maximality}
There are many other ways to understand a tri-pant; we provide an example below with an alternate, less restrictive definition.

\begin{defn}\label{def: tripantsII}
A {\em supercollection} is a maximal collection of homotopy classes of non-separating essential simple closed curves on $\twicepuncturedtorus$ that pairwise intersect at most once. 
\end{defn}

\begin{lem}\label{lem: leq6}
If $X$ is a supercollection, then $card(X)\leq 6$.
\end{lem}
\begin{proof}
Recall the inclusion map $\eta: T^2\setminus\{\circ,\bullet\} \to T^2$ and its induced homomorphism $\eta_*$. By definition, we have that $\eta_*(\alpha) = [\eta \circ \alpha] \subset T^2$ for every $\alpha \subset \twicepuncturedtorus$. Now for the sake of contradiction, suppose that $card(X)\ge 7$. By Lemma~\ref{lem:intleq2}, we know that the images of the elements in $X$ under $\eta_*$ pairwise intersect at most once on $T^2$. It is well-known \cite{MC} that disjoint, essential simple closed curves on $T^2$ are homotopic. Another result in \cite{MC} proves that there can be at most three homotopy classes of simple closed curves on the torus that pairwise intersect exactly once. Thus, the homotopy classes in $X$ map into at most three homotopy classes in $T^2$ under $\eta_*$. It follows that there is at least one class of curves on $T^2$ whose pre-image under $\eta_*$ has at least three elements. However, we note that this contradicts Lemma~\ref{lem: cardleq2} and can therefore conclude that $card(X)\le 6$.
\end{proof}

\begin{lem}\label{lem: geq6}
If $X$ is a supercollection, then $card(X)\ge 6$.
\end{lem}
\begin{proof}
For the sake of contradiction, assume that $card(X)<6$. Regarding the intersection pattern of the elements in $X$, we have two cases to consider: \begin{itemize}
    \item Suppose there is no class in $X$ that intersects every other class in $X$ exactly once, meaning that every element in $X$ is disjoint from at least one other element. Then Lemma~\ref{lem: cardleq2} allows us to conclude that either $card(X) =2$ or $card(X) =4$ and we can pair these elements up with their disjoint ``partner" so that $X = \{\alpha, \alpha'\}$ or $X = \{ \{\alpha, \alpha'\}, \{\beta,\beta'\}\}$ where $\eta_*(\alpha) = \eta_*(\alpha')$ and $\eta_*(\beta) = \eta_*(\beta')$. Thus, we see that $\{\alpha, \alpha'\}$ and $\{\beta, \beta'\}$ determine pants decompositions of $\twicepuncturedtorus$ because they separate $\twicepuncturedtorus$ into two copies of a once-punctured annulus. So, $\{\alpha, \alpha'\}, \{\beta, \beta'\} \in \mathcal{P}$. Consider the map $\Phi$, as defined in Definition \ref{lem: alpha in P}, and recall that it is bijective. Then $\Phi^{-1}(\{\alpha,\alpha'\}) = a$ and $\Phi^{-1}\{\beta,\beta'\} = b$ where $a,b \in \mathcal{A}$ and intersect only at $\bullet$. 
    
    Now, we consider the two possible cases for $X$. First, suppose that $X= \{\alpha, \alpha'\}$. Thus, we know that the elements of $X$ can be represented as $a\in \mathcal{A}$, and we can find some $b \in \mathcal{A}$ so that there exist representatives $\hat{a},\hat{b}$ which intersect only at $\bullet$. Thus, we have that $\Phi(b) = \{\beta,\beta'\} \in \mathcal{P}$, and $\beta$, $\beta'$ extend $X$. Next, suppose that $X = \{ \{\alpha, \alpha'\}, \{\beta,\beta'\}\}$. We know that the elements of $X$ can be represented as $a,b \in \mathcal{A}$ that intersect each other only at $\bullet$. Thus, we know that we can find another element, $c \in \mathcal{A}$, that extends $a,b$ to a tri-arc. So, we take $\Phi(c) = \{\gamma, \gamma'\}$ and see that $\gamma, \gamma'$ extend $X$.
    Thus, we see that in both cases, $X$ is not maximal.

    \begin{figure}[htbp]
        \centering
    \begin{overpic}
    [width=8cm]{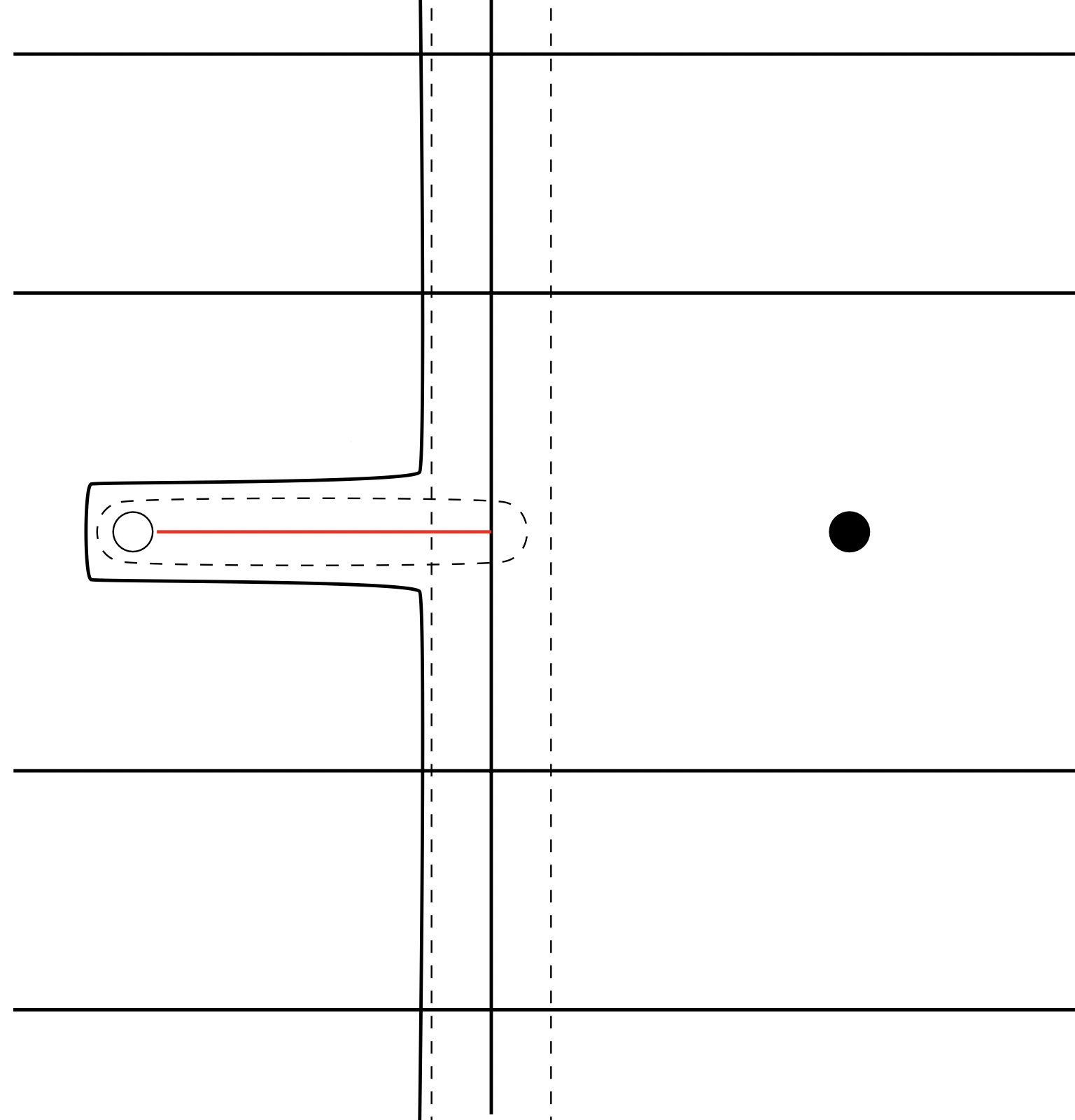}
    \put(10,90){$\alphahat$}
    \put(10,68){$\alphahat'$}
   \put(10,25){$\betahat$}
    \put(10,3){$\betahat'$}
    \put(42,103){$\gammahat$}
    \put(32,61){$\gammahat'$}

    \end{overpic}
    
        \caption{Constructing $\gamma'$ by navigating around $K$, pictured in red.}
        \label{fig: kanddelta}
    \end{figure}
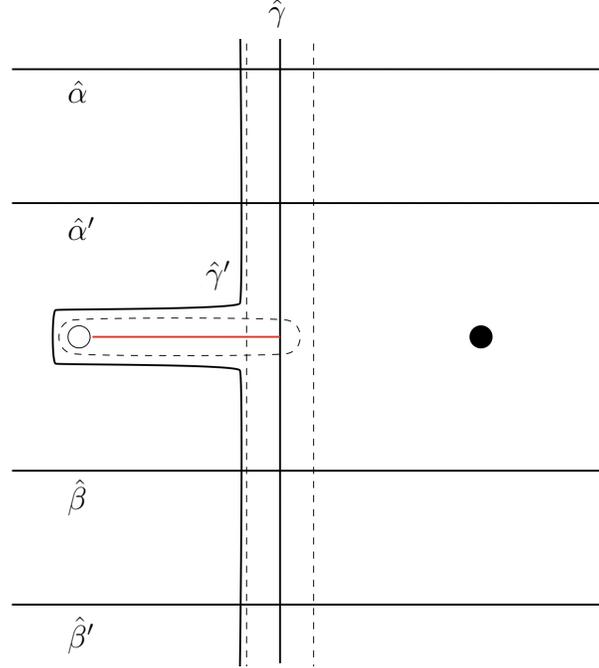
    
    \item Suppose there exists a class $\gamma \in X$ that intersects every other class in $X$ exactly once. We will first focus on the case where $card(X)=5$. We know from Lemma~\ref{lem: cardleq2} that every $X$ with $card(X) =5$ will be in the form $X = \{\alpha, \alpha', \beta, \beta', \gamma\}$, where $\{\alpha,\alpha'\}, \{\beta,\beta'\}$ are disjoint pairs of homotopy classes that form pants decompositions and $\gamma$ intersects every other element in $X$ exactly once. We choose respective representatives $\alphahat, \alphahat', \betahat,\betahat',\gammahat$ to be in minimal position. We claim that there is a simple path $K:[0,1]\to\oncepuncturedtorus$ that satisfies the following conditions: that (1) $K\vert_{[0,1)}\subset \twicepuncturedtorus$; that (2) $K(0)\in \gammahat$ and $K(1) \in \{\bullet, \circ\}$; and that (3) $K$ is disjoint from $\alphahat$, $\alphahat'$, $\betahat$, and $\betahat'$. To see this, we show that such a $K$ exists by on cutting along one element in $\{\alphahat,\alphahat'\}$ and one in $\{\betahat,\betahat'\}$. Without loss of generality, we cut along $\alphahat,\betahat$ to yield a square with opposite sides identified, containing both punctures and the remaining representatives $\alphahat'$ and $\betahat'$ in its interior. Therefore, there exist 4 different cases for how $\gammahat$ could appear within this twice-punctured square up to homotopy, as shown below in Figure~\ref{fig:K}. 

\begin{figure}[htbp]
    \centering
    \begin{overpic}
    [width=8cm]{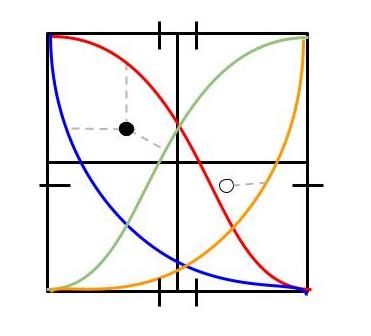}
    \put(46,87){$\alphahat$}
    \put(46,5){$\alphahat$}
    \put(63,49){$\alphahat'$}
    \put(5,46){$\betahat$}
    \put(86,46){$\betahat$}
    \put(50,70){$\betahat'$}
    \put(17,65){$\gammahat_1$}
    \put(74,65){$\gammahat_2$}
    \put(20,20){$\gammahat_3$}
    \put(72,20){$\gammahat_4$}
    \end{overpic}
    \vspace{0.3cm}
    \caption{There are four possibilities for $\gammahat$. Such a $K$ exists in each case.}
    \label{fig:K}
\end{figure}
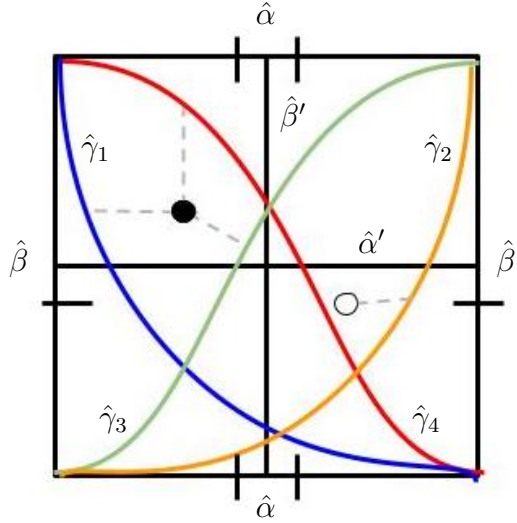

    We can see that for every $i\in\{1,2,3,4\}$, there exists a simple path $K$ such that $K(0)\in \hat{\gamma}_i$ and $K(1) \in \{\bullet, \circ\}$. Note also that $K|_{[0,1)}$ is disjoint from $\alphahat\cup\alphahat'\cup\betahat\cup\betahat'\cup\hat{\gamma}$ and lies inside $\twicepuncturedtorus$. The possible paths are shown by the grey dotted lines in Figure~\ref{fig:K}. The same strategy can be applied to the cases in which $card(X)\le 4$ and there exists a $\gamma\in X$ such that $\gamma$ intersects every other class in $X$ exactly once. In those cases, the smaller number of curves in $X$ provides fewer obstructions to the existence of a path $K$ satisfying the necessary conditions. Therefore, we conclude that such a $K$ exists for any supercollection $X$ with $card(X) < 6$. 
    
    We now use the existence of such a $K$ to show that $X$ is not maximal. To achieve this result, construct tubular neighborhoods $V_{\hat{\gamma}}$, $V_K$ about $\hat{\gamma}$ and $K$ such that their boundaries are disjoint from every other curve. Then we can construct a homotopy class $\gamma'$ with representative $\hat{\gamma}$ by following along the boundaries of $V_K$ and $V_{\hat{\gamma}}$, as seen in Figure~\ref{fig: kanddelta}. By construction, we see that $\hat{\gamma}'$ and $\hat{\gamma}$ bound two once-punctured annuli, and $\hat{\gamma}'$ intersects every other class, $\alpha,\alpha',\beta,\beta'$ exactly once. It follows that $\gamma'$ is a class of essential, non-separating simple closed curves intersecting every other class in $X$ at most once. Thus, $X$ is not maximal. 
\end{itemize} We conclude that $card(X)\ge6$.
\end{proof}

\begin{cor}\label{cor: cardis6}
If $X$ is a supercollection, then $card(X)=6$.
\end{cor}

\begin{manualtheorem}{\ref{thm: tripants and maximal collections}}
A set $X$ of homotopy classes of simple closed curves on $\twicepuncturedtorus$ is a tri-pant if and only if $X$ is a supercollection.
\end{manualtheorem}

\begin{proof}
We will prove the given statement in both directions.
\begin{itemize}
\item[($\Leftarrow$)] Let $X$ be a supercollection. We will show that $X$ is a tri-pant. 
By Corollary~\ref{cor: cardis6}, we know that $card(X)=6$. So, let $X = \{\alpha_1, \alpha_2, \alpha_3, \alpha_4, \alpha_5, \alpha_6\}$ and denote the representatives of each class by $\alphahat_i$ for $i\in\{1,2,3,4,5,6\}$, chosen so that they intersect one another in minimal position. Recall that since $X$ is a supercollection, we have $i(\alpha_i, \alpha_j)\le1$ for all $i,j\in\{1,2,3,4,5,6\}$. For the sake of contradiction and without loss of generality, assume that $\alpha_6$ intersects every other representative exactly once; that is, $i(\alpha_6,\alpha_i)=1$ for all $i\neq6$. By Lemma~\ref{lem:intleq2}, we know that $i(\eta_*(\alpha_6), \eta_*(\alpha_i)) = 1$ on $T^2$ for all $i\neq6$. Recall from \cite{MC} that a maximal collection of essential simple closed curves that intersect pairwise at most once on $T^2$ has three elements. Then without loss of generality, we have $\eta_*(\alpha_1) \simeq \eta_*(\alpha_2) \simeq \eta_*(\alpha_3)$. This contradicts Lemma \ref{lem: cardleq2} because we found 3 disstinct elements of $X$ that are sent to the same homotopy class under $\eta_*$. Therefore, $\alphahat_6$ must be disjoint from at least one other $\alphahat_i$. We can extend this conclusion further, though. Since a maximal collection of disjoint essential non-separating simple closed curves in $\twicepuncturedtorus$ has cardinality equal to 2, we conclude that each $\alpha_i$ must be disjoint from exactly one other representative. Then $X$ is a set of three pairs of classes of disjoint, non-separating essential simple closed curves with representatives that pairwise intersect every other representative exactly once (using the alternate definition of a pants decomposition). Thus, $X$ is a tri-pant. 
\item[($\Rightarrow$)] Let $X$ be a tri-pant. If $X$ were not maximal, then we could find a set of homotopy classes of simple closed curves that pairwise intersect at most once with cardinality 7; Lemma~\ref{lem: leq6} rules this case out. Then $X$ is maximal and thereby is a supercollection.
\end{itemize} We have thus proven the statement in both directions.
\end{proof}
\subsection{The inclusion map and rational numbers}\label{sec: inclusion} Can we find an even more accessible representation of tri-pants that retains bijectivity? The answer to this question appears to be ``no"; however, we can describe a beautiful relationship between tri-pants and the rational numbers using the inclusion map. Let $\eta: \oncepuncturedtorus\to T^2$ be the inclusion map. We investigate its induced homomorphism $\eta_*:\pi_1(\oncepuncturedtorus, \bullet)\to\pi_1(T^2,\bullet)$, or perhaps more simply, $$\eta_*: \Z\cdot a\ast \Z\cdot b \to \Z\times\Z.$$ We note that $\eta_*$ maps ``words" written with the letters $a, a^{-1}, b$, and $b^{-1}$ into tuples of integers $(p,q)$, where $p$ and $q$ correspond to the number of times a given curve travels about the meridian and longitude of the torus, respectively. 

We see that this map is not bijecive. For example, we know that $aba^{-1}$ and $b$ are distinct elements of $\pi_1(\oncepuncturedtorus, \bullet)$, but $\eta_*(aba^{-1})=\eta_*(b)=(0,1)$. We should not let this lack of bijectivity discourage us though. If, instead of the tuple $(p,q)$, we represent elements in the image of $\eta_*$ by the rational numbers $\frac{p}{q}$, we introduce a powerful connection between tri-pants and $\Q$. In the next section, we thoroughly investigate this correspondence and extend it to relate the tri-pants graph with the dual of the Farey complex.
\section{A Tale of Two Graphs}\label{sec: graphs}
\subsection{The Farey graph and its dual} \label{subsec:farey}
As shown in the previous section, we can relate curves on  $\twicepuncturedtorus$ to those on $T^2\setminus\{\circ\}$ via the inclusion map, $\eta$. This relationship will provide implications for our main results, as set the homotopy classes of essential curves in $\oncepuncturedtorus$ is in one-to-one correspondence with the vertices of the Farey complex. This graph, denoted $\mathcal{F}$, provides a tessellation of the upper half-plane model of the hyperbolic plane, $\mathbb{H}^2$, which corresponds to the union of the set of complex numbers that have positive imaginary part with the set containing only infinity. 

\begin{defn}
The {\em Farey graph} is contained in $\overline{\mathbb{H}^2}$, with vertex set $V=\Q\cup\{\infty\}$. The points $\frac{p}{q}$ and $\frac{p'}{q'}$ are connected by an edge if and only if $|pq'-p'q|=1$. If $q=0$, then $\frac{p}{q}=\infty$.
\end{defn}

\begin{figure}[htbp]
     \centering
     \includegraphics[width=8cm]{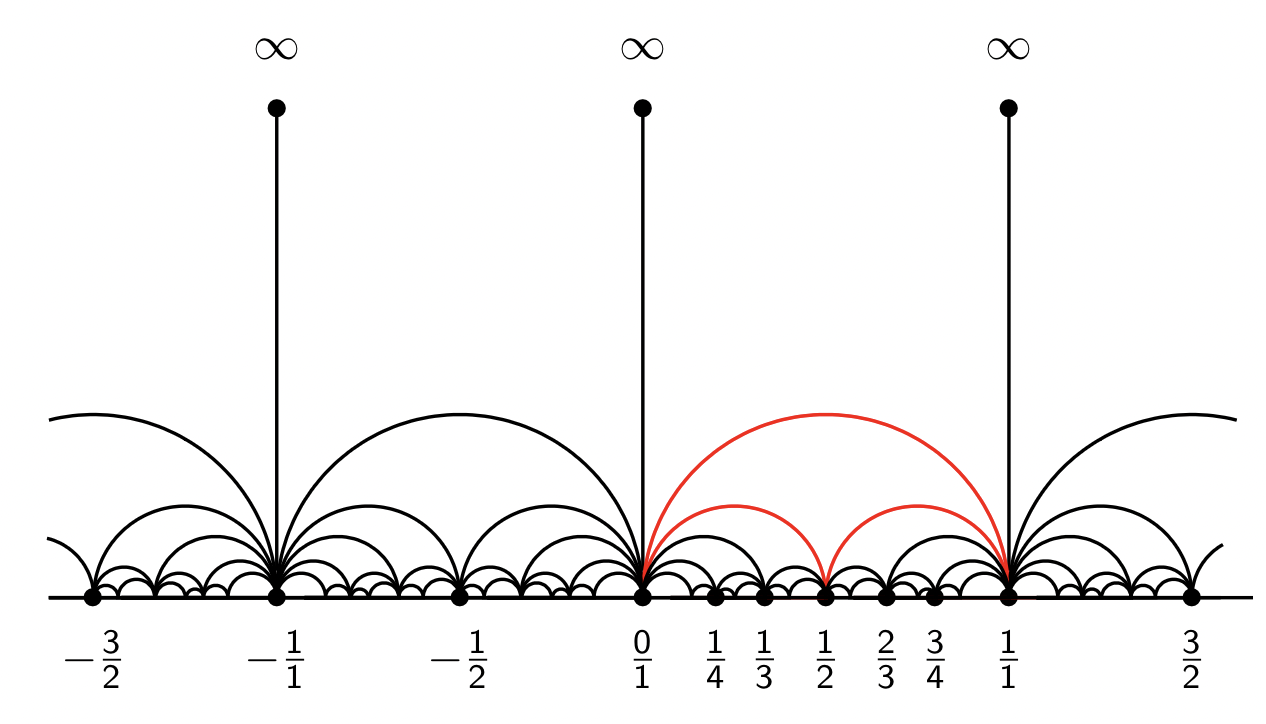}
    \caption{The Farey graph.}
    \label{fig:fareyGraph}
\end{figure}

If three vertices in $\mathcal{F}$ are pairwise connected by an edge, then those vertices create a \textit{triangle} in the Farey graph, as seen by the red edges in Figure \ref{fig:fareyGraph}. We say that two triangles are \textit{adjacent} if they share an edge. We define a metric $d_\mathcal{F}(\frac{p}{q},\frac{p'}{q'})$ to be the minimum number $n\in \Z$ of consecutive edges connecting $\frac{p}{q}$ and $\frac{p'}{q'}$. It is simple to see that the Farey Graph is connected.

As discussed in Section~\ref{sec: inclusion}, the inclusion map induces a homomorphism $\eta_*$ that allows us to associate pairs of integers with simple closed curves in $T^2\setminus\{\circ\}$. Then, two vertices in $\mathcal{F}$ are connected by an edge if the associated simple closed curves intersect exactly once. The relationships between the Farey graph and simple closed curves in $T^2\setminus\{\circ\}$ are as follows: 
\begin{itemize}
    \item Let $\alpha \in \pi_1(T^2\setminus\{\circ\}, \bullet)$ such that $\eta_*(\alpha) = (p,q)$.  Then $\alpha$ relates to the vertex $\frac{p}{q}\in\mathcal{F}$. It follows that if $\eta_*(\alpha) = (1,0)$, then $\alpha$ relates to the point at infinity in $\mathcal{F}$. Note that for all $(p,q)\in \Z\times \Z$ with $p,q$ coprime, there exists a unique $\alpha\in \pi_1(\oncepuncturedtorus, \bullet)$ such that $i_*(\alpha)=(p,q)$ and $\alpha$ is represented as a simple closed curve.  
    \item The map $\eta_*$ maps tri-arcs to triangles in $\mathcal{F}$. We see that applying $\eta_*$ to the arcs of a tri-arc yields three adjacent vertices on $\mathcal{F}$, due to the condition that the arcs intersect pairwise once. Thus, each pair of vertices in $\mathcal{F}$ is connected by an edge, so the tri-arc corresponds to a unique triangle in $\mathcal{F}$. We note that adjacent triangles in $\mathcal{F}$ correspond to tri-arcs that differ by exactly one arc. This conclusion follows from the fact that they share two vertices on $\mathcal{F}$ when applying $\eta_*$ to their arcs. 
\end{itemize}
\begin{defn}
Two tri-arcs are \emph{adjacent} if they differ by exactly one arc. 
\end{defn}
We will now extend this connection between homotopy classes of essential simple closed curves on $T^2\setminus\{\circ\}$ and $\mathcal{F}$ to a correspondence between curves on $T^2\setminus\{\circ\}$ and the dual of the Farey complex, denoted $\fareydual$. 
\begin{figure}[htbp]
    \centering
    \includegraphics[width=8cm]{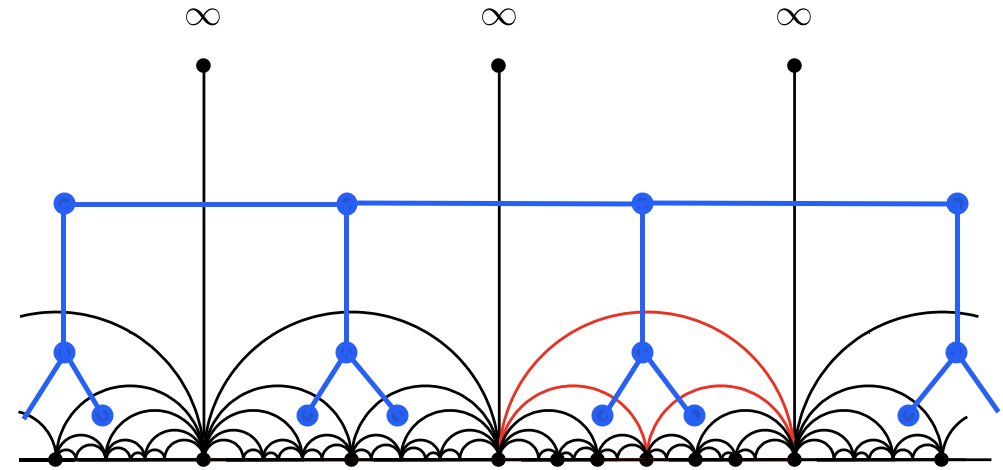}
    \caption{The dual of the Farey graph (blue) is superimposed atop $\mathcal{F}$.}
    \label{fig:FareyDual}
\end{figure}
\begin{defn}
The {\em dual of the Farey graph}, denoted $\fareydual$, is the graph associated $\mathcal{F}$ having a vertex $T^*$ per each triangle $T$ in $\mathcal{F}$ and whose edges connect two vertices $v$, $v'\in \fareydual$ if and only if their corresponding triangles $T^*$ and $(T^*)'$ are adjacent in $\mathcal{F}$.
\end{defn}
It is important to note that the dual of the Farey graph is a connected, trivalent, infinite tree and therefore has infinite diameter. For the remainder of this paper, we will examine powerful connections between $\mathcal{F}^*$ and the tri-pants graph.

\subsection{The Tri-Pants Graph}\label{sec:moves}
In this section, we will define, construct, and build intuition on the tri-pants graph. We will then state and prove the central results of this paper; namely, that the tri-pants graph is connected and has infinite diameter. First, though, one must understand the structure of the graph itself.
\subsubsection{Elementary moves}\label{subsubsec:elementary moves}
As stated in the previous section, tri-arcs that share two of their arcs that project into in adjacent triangles in $\mathcal{F}$. We now seek to understand how adjacent tri-arcs relate to one another. In fact, there are two distinct operations for moving between adjacent tri-arcs; tri-pants that correspond to adjacent tri-arcs differ by an elementary move. 

\begin{defn}[Elementary Moves]
Let $T_*=\{a,b,c\}$ be a tri-arc in $\twicepuncturedtorus$, and choose $\hat{a}, \hat{b}$, and $\hat{c}$ to be respective representatives of $a, b,$ and $c$. Without loss of generality, cut along $\hat{a}$ and $\hat{b}$ to obtain a punctured square with opposite sides identified, as pictured in Figure \ref{fig: punctured square}. Then $\hat{c}$ will lie inside this square and connect two diagonal vertices. We now define the two possible operations one may perform on $c$:
\begin{itemize}
    \item  Applying a \emph{big flip} of $T_*$ on $c$ yields a tri-arc $T'_*=\{a,b,c'\}$, where $c'\in\A$ is the unique homotopy class of arcs $\hat{c}'$ in the punctured square created by cutting along $\hat{a}$ and $\hat{b}$ such that $\hat{c}'$ and $\hat{c}$ share the same endpoints but lie on opposite sides of the puncture $\circ$. This operation is pictured in Figure~\ref{fig: big flip of c}.
    \item Applying a \emph{small flip} of $T_*$ on $c$ yields a tri-arc $T''_*=\{a,b,c''\}$, where $c''\in\A$ is a homotopy class of arcs $\hat{c}''$ in the punctured square created by cutting along $\hat{a}$ and $\hat{b}$ such that $\hat{c}''$ and $\hat{c}$ do not share endpoints. There are two distinct homotopy classes that applying a small flip of $T_*$ on $c$ may yield. This operation is pictured in Figure \ref{fig:small_flip}.
\end{itemize}
\end{defn}

\begin{figure}[htbp]
    \centering
    \includegraphics[scale=0.8]{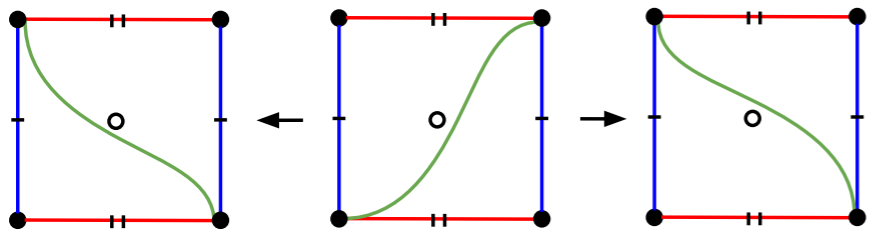}
    \caption{Small flip on a tri-arc.}
    \label{fig:small_flip}
\end{figure}

At times, it may be useful to represent tri-arcs as triples of elements in the fundamental group $\pi_1(\oncepuncturedtorus,\bullet)$. Since the homotopy classes $a,b,c$ consist of unoriented arcs, we must (arbitrarily) assign an orientation to each homotopy class to obtain elements of the fundamental group. This requirement motivates us to define a tri-arc with orientation.
\begin{figure}[htbp]
    \centering
    \begin{overpic}
    [scale=.6]{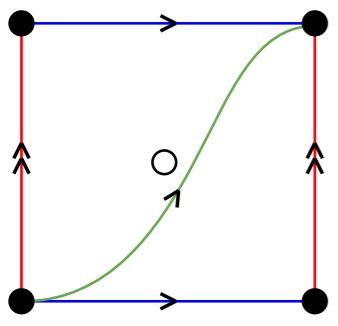}
    \end{overpic}
    \put(-54,-4){$a$}
    \put(-107,46.5){$b$}
    \put(-44,34){$c$}
    \caption{An oriented tri-arc $T_{\pi_1}=\{a,b,c\}$.}
    \label{fig: oriented triarc}
\end{figure}
\begin{defn}\label{defn: oriented triarc}
An {\em oriented tri-arc} $T_{\pi_1}=\{a,b,c\}$ is a tri-arc with choices of orientation assigned to the homotopy classes $a,b,c$ so that $a, b, c\in \pi_1(\oncepuncturedtorus,\bullet)$.

\end{defn}

The main difference between a standard tri-arc $T_*$ and an oriented tri-arc $T_{\pi_1}$ is that $T_*\subset\A$ but $T_{\pi_1}\subset\pi_1(\oncepuncturedtorus,\bullet)$. If we apply a big flip of $T_{\pi_1}=\{a,b,c\}$ on $c$, then the resulting arc $c'$ and $c$ will share the same orientation, as shown in Figure~\ref{fig: big flip of c}. We now define how a big flip operates on an oriented tri-arc through an algebraic lens.  

\begin{figure}[htbp]
    \centering
    \begin{overpic}[scale=.6]{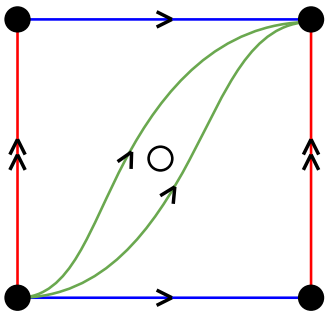}
    \end{overpic}
    \put(-53,-5){$a$}
    \put(-105,44.5){$b$}
    \put(-44,34){$c$}
    \put(-70,52){$c'$}
    \caption{Applying a big flip of $T_{\pi_1}=\{a,b,c\}$ on $c$ yields $c'$.}
    \label{fig: big flip of c}
\end{figure}
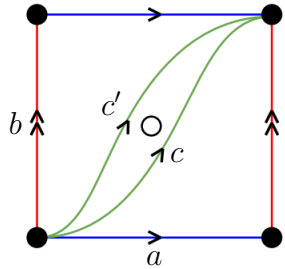

\begin{lem}\label{conj: cyclicperm}
Let $T_{\pi_1}=\{a, b, c\}$ be an oriented tri-arc. Fix $x,y\in\{a,a^{-1},b,b^{-1}\}$ in such a way that $c=xy$. Applying a big flip of $T_{\pi_1}$ on $c$ yields $c'$, and we denote this operation by $B_{c}(T_{\pi_1})$. Algebraically, we have $c':=[y,x]c$, so $B_{c}(T_{\pi_1}) =  \{a,b,c'\} = \{a,b,[y,x]c\}$.
\end{lem}
\begin{proof}
Recall that Lemma~\ref{lem: tri-arc decomp} allows us to always let $c=xy$ as desired. Choose respective representatives $\hat{a},\hat{b},\hat{c}$ of $a,b,$ and $c$ that intersect only at $\bullet$ and create a punctured square by cutting $\hat{a}$ and $\hat{b}$. We may assume orientations and labeling as in Figure~\ref{fig: oriented triarc}, since the punctured square is invariant under these choices and thus every other case is identical up to relabeling. Thus $c=ab$ and applying a big flip on $c$ yields the arc $c'$, as pictured in Figure~\ref{fig: big flip of c}. Notice that $c\cdot(c')^{-1} = aba^{-1}b^{-1}$, so $c' = bab^{-1}a^{-1}c = [b,a]c$. Therefore, $B_{c}(T_{\pi_1}) = \{a,b,[y,x]c\}$.
\end{proof}

By defining this algebraic relationship between tri-arcs differing only by a big flip, we are able to prove the following properties.

\begin{lem}\label{lem: twoflipID}
If $T_{\pi_1}$ contains the arc $c$, then $B_{c'}(B_c(T_{\pi_1})) = T_{\pi_1}$ with $c'$ as in Lemma~\ref{conj: cyclicperm}.
\end{lem}

\begin{proof}
Let $T_{\pi_1}=\{a,b,c \}$, where $c=ab$ as stated in Lemma~\ref{conj: cyclicperm}. We can compute that $B_c(T_*) = \{a, b, [b,a]c\} = \{a,b, ba\}$. It follows from this point that $$B_{c'}(B_c(T_{\pi_1})) = \{a,b, [a,b]ba\} = \{a,b,ab\} = \{a, b, c\}.$$ Since the remaining cases for $c$ (in terms of $a,b,a^{-1},b^{-1}$) are identical up to relabeling, we conclude that applying two big flips in succession to the same arc preserves $T_{\pi_1}$.
\end{proof}

\begin{lem}
The composition of big flips along distinct arcs does not commute.
\end{lem}

\begin{proof}
Let $T_{\pi_1}=\{a, b, c\}$ be a tri-arc representation of a tri-pant. We will show that $B_{b}(B_a(T_{\pi_1})) \neq B_a(B_b(T_{\pi_1}))$. By Lemma~\ref{lem: tri-arc decomp}, we can assume without loss of generality that $a = cb$, and thus $b = c^{-1}a$. First, we see that $B_a(T_{\pi_1}) = \{[b,c]a, b, c\} = \{bcb^{-1}c^{-1}cb,b,c\} = \{bc,b,c\}$. It follows that \begin{align*}
    B_b(B_a(T_{\pi_1})) &= \{bc, [c^{-1}, bc] b, c\} \\
    &= \{bc,c^{-1}bc(cc^{-1})(b^{-1}b),c\} \\
    &= \{bc, c^{-1}bc, c\}.
\end{align*} We also have
$B_b(T_{\pi_1}) = \{a, [a,c^{-1}]b, c\} = \{a, ac^{-1}a^{-1}c(c^{-1}a), c\} = \{a, ac^{-1}, c\}$. Then \begin{align*}
    B_a(B_b(T_{\pi_1})) &= \{[c,ac^{-1}]a, ac^{-1},c\}\\ &=\{cac^{-1}(c^{-1}c)(a^{-1}a), ac^{-1}, c\}\\ 
    &=\{cac^{-1}, ac^{-1},c\}.
\end{align*}
Substituting $a = cb$ yields $B_a(B_b(T_{\pi_1}))=\{ccbc^{-1}, cbc^{-1},c\}$. Then even under the inverses of elements, we see that $B_{b}(B_a(T_{\pi_1})) \neq B_a(B_b(T_{\pi_1}))$.
\end{proof}

\subsubsection{Construction of the tri-pants graph}
Now that we understand the structure of tri-pants, tri-arcs, and elementary moves, we will construct the tri-pants graph, whose properties constitute the remainder of this paper. We denote the tri-pants graph $\tripantsgraph$. 
\begin{defn}\label{def: tripantsgraph}
    The \emph{tri-pants graph} $\tripantsgraph$ is a graph with vertices corresponding to choices of tri-pants on $\twicepuncturedtorus$. Two vertices are connected by an edge if and only if the associated tri-pants differ by an elementary move.
\end{defn}
We wish to understand the geometric properties of this graph in order to learn more about how tri-pants behave on $\twicepuncturedtorus$. Thus, we will observe the graph's structure, connectedness, and diameter. The results of Lemma~\ref{clm:nottree} and Lemma~\ref{lem: nineadjacent} follow quickly from Definition~\ref{def: tripantsgraph}. 

\begin{lem}\label{clm:nottree}
The tri-pants graph is not a tree.
\end{lem}
\begin{proof}
See Figure~\ref{fig:loopongraph} for an example of a closed loop in $\tripantsgraph$.
\end{proof}
\begin{figure}[htbp]
    \centering
    \includegraphics{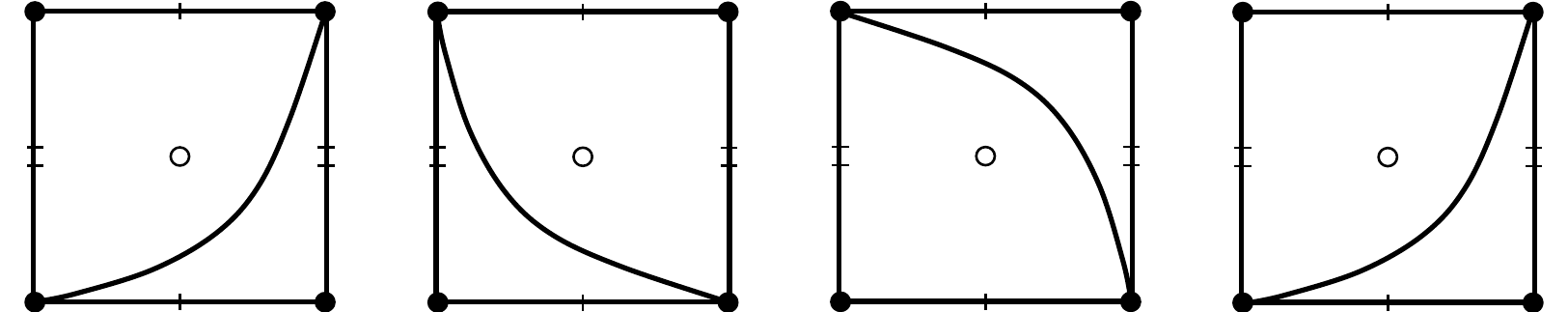}
    \caption{A closed loop in $\tripantsgraph$ created by applying a small flip, a big flip, and a different small flip to the diagonal in the left image.}
    \label{fig:loopongraph}
\end{figure}
\begin{lem}\label{lem: nineadjacent}
Let $T$ be a tri-pant, so $T$ is a vertex of $\tripantsgraph$. Then $\deg(T)=9$.
\end{lem}

\begin{proof}
Let $T$ correspond to $T_{*}=\{a, b, c\}$ with respective representatives $\hat{a}, \hat{b}$, and $\hat{c}$. There are three possible choices one can make to represent $T_{*}$ on a punctured square with arbitrarily-chosen orientation. \begin{enumerate}
    \item Cut along $\hat{a}$ and $\hat{b}$ to yield the construction shown in Figure~\ref{fig: oriented triarc}. Since $c = ab$, we write that $T_{\pi_1}=\{a,b,ab\}$. There are three possible elementary moves to apply on $ab$: two small flips and a big flip. As elements in $\pi_1(\oncepuncturedtorus,\bullet)$, these operations yield the triples $\{a,b,ba\}$, $\{a,b,ab^{-1}\}$, and $\{a,b,ba^{-1}\}$.
    \item Cut along $\hat{a}$ and $\hat{c}$. Applying the three possible elementary moves on $b$ yields the triples $\{a,aba, ab\}$, $\{a, a^2b, ab\}$, and $\{a,ab^{-1}a^{-1}, ab\}$, since $c = ab$. 
    \item Cut along $\hat{b}$ and $\hat{c}$. Applying the three possible elementary moves on $a$ yields the triples $\{bab,b,ab\}$, $\{ab^2,b,ab\}$, and $\{b^{-1}ab,b,ab\}$, since $c = ab$. 
\end{enumerate}
Even under the inverses of elements, these nine resulting triples are distinct from one another, allowing us to conclude that $\deg(T)=9$.
\end{proof}

\subsubsection{The projection map $\pi$}\label{sec: pi}
We now use $\fareydual$ to produce more substantial results about the geometric structure of $\tripantsgraph$. We must first define a relation between the two graphs. 

\begin{defn}\label{defn: pitilde}
Define the projection map $\pi: \tripantsgraph \to \fareydual$ such that: 
\begin{itemize}
    \item Every $T\in \tripantsgraph$ maps to a vertex $v\in \fareydual$ corresponding to the triangle with vertices obtained by applying $i_*$ on the arcs of $T_*$, the tri-arc representing $T$. Observe in particular that (1) if $T$ and $T'$ differ by a big flip, then $\pi(T) = \pi(T')$, and that (2) if $T$ and $T'$ differ by a small flip, then $\pi(T)$ and $\pi(T')$ are adjacent in $\fareydual$.
    \item Let $e \in \tripantsgraph$ be an edge connecting two vertices $T, T'$ that differ by a \emph{big flip}. Then $\pi(e)=\pi(T) = \pi(T')$.
    \item Let $e \in \tripantsgraph$ be an edge connecting two vertices $T, T'$ that differ by a \emph{small flip}. Then $\pi(e)$ connects $\pi(T)$ to $\pi(T')$. 
\end{itemize}
\end{defn}

\begin{figure}[htbp]
    \centering
    \includegraphics[width = 14cm]{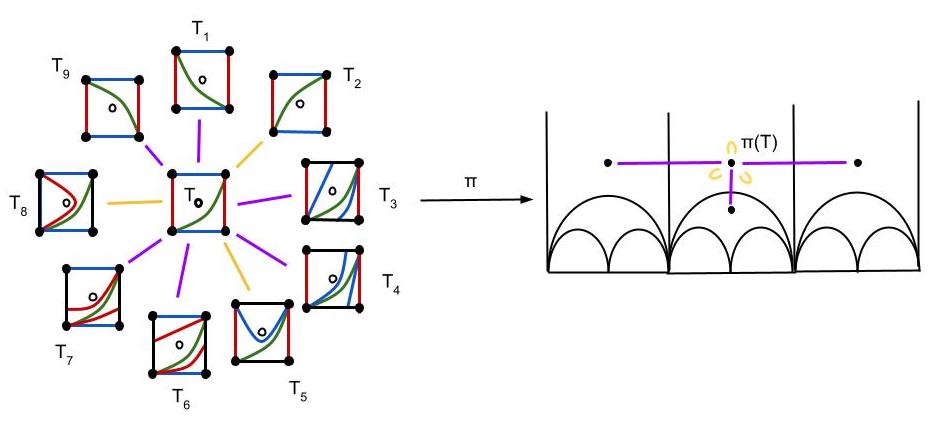}
    \caption{The map $\pi$ maps vertices in $\tripantsgraph$ to vertices in $\fareydual$.}
    \label{fig:pimap}
\end{figure}

Figure \ref{fig:pimap} provides intuition as to how the tri-pants graph is projected onto $\fareydual$. In this image, we represent edges in $\tripantsgraph$ corresponding to small flips as purple and edges in $\tripantsgraph$ corresponding to big flips as yellow. As proven in Lemma~\ref{lem: nineadjacent}, we know that vertices in $\tripantsgraph$ have valence equal to 9. However, we can see that their images under $\pi$ only have valence 3. Thus, we desire to understand what lies in the preimage of each element of $\fareydual$ under this map. 

\begin{defn}
Define a \emph{fiber} of $\pi$ to be $\pi^{-1}(v)\subset \tripantsgraph$ for some $v\in \fareydual$. 
\end{defn}

By Definition~\ref{defn: pitilde}, if we have $T\in\pi^{-1}(v)$ for some $v\in\fareydual$, then every other tri-pant $T'$ which differs from $T$ by a finite sequence of big flips will also be in the same fiber $\pi^{-1}(v)$. In Figure~\ref{fig:pimap}, we can see that $T$, $T_2$, $T_5$, $T_7$ and the yellow edges that connect them all exist within the same fiber. We must also note that for every edge $e \in \fareydual$, we see that $card(\pi^{-1}(e)) \geq 2$. To convince ourselves of this fact, look again to Figure~\ref{fig:pimap}. As shown, $T_8$ and $T_9$ differ from $T$ by a small flip and from each other by a big flip. Then $\pi(T_8) = \pi(T_9)$ and the two tri-pants exist in the fiber adjacent to that of $T$. Thus, we see that the two purple edges joining $T_8$ and $T_9$ to $T$ in $\tripantsgraph$ have the same image under $\pi$. Lemma~\ref{lem: pisurj} follows naturally from this observation.

\begin{lem}\label{lem: pisurj}
The map $\pi$ is surjective.
\end{lem}
\subsubsection{Connectedness}\label{sec: connect}

With our goal being to understand how two tri-pants can relate, it is important to see if we can use elementary moves to change any tri-pant into another. This problem translates into proving one of our main results, that the tri-pants graph $\tripantsgraph$ is connected. To tackle this proof, we will focus on the study of big and small flips. Because big flips connect two tri-arcs within the same fiber and small flips connect two tri-arcs within different fibers, it is sufficient to prove the following statements:
\begin{enumerate}
    \item Proving that any two tri-arcs in a fiber can be related by a sequence of big flips.
    \item Proving that any two fibers are connected via a sequence of small flips (and possibly big flips within intermediate fibers).
\end{enumerate}
It turns out that proving $(2)$ is not such a difficult task, given the relationship that the map $\pi$ gives between the set of fibers and the vertices of $\fareydual$. However, proving $(1)$ is more difficult and requires us to define homeomorphisms of $\twicepuncturedtorus$ that perform big flips on tri-arcs. To locate these homeomorphisms we look within the surface's mapping class group.

\begin{defn}\label{defn: pmod}
Consider a surface with punctures $S=(\overline{S},\mathcal{P})$ and let Homeo$^+(S,\partial S\cup \mathcal{P})$ be the set of orientation-preserving homeomorphisms $f:\overline{S}\to\overline{S}$ such that $f|_{\partial S\cup\mathcal{P}}$ is the identity map. Additionally, let Homeo$_0(S,\partial S\cup \mathcal{P})$ be the set of maps $f\in$ Homeo$^+(S,\partial S\cup \mathcal{P})$ for which there exists a homotopy $F:[0,1]\times \overline{S}\to\overline{S}$ from $f$ to the identity maps and so that $F(t,\cdot) \in$ Homeo$^+(S,\partial S\cup \mathcal{P})$ for all $t\in[0,1]$ (that is, $f$ is isotopic to the identity in Homeo$^+(S,\partial S\cup \mathcal{P})$). We define the \emph{pure mapping class group}
$$\text{PMod}(S):=\faktor{\text{Homeo}^+(S,\partial S\cup \mathcal{P})}{\text{Homeo}_0(S,\partial S\cup \mathcal{P})}.$$
\end{defn}

From the definition, we see that the elements of PMod$(S)$ are isotopy classes of orientation-preserving homeomorphisms on $S$ which fix the boundary points and punctures. Within PMod$(\twicepuncturedtorus)$, we wish to find the isotopy classes that rearrange tri-arcs via big flips. Luckily, there exist homeomorphisms that act nicely for our study.

\begin{defn}\label{defn: push}
For any homotopy class $\theta\in\pi_1(\oncepuncturedtorusblack,\circ)$, define an isotopy class Push$(\theta)\in\pmodtorus$ of maps $P_a\in\homeoplus$ for $a\in\theta$. To apply a map $P_a$, we slide $\circ$ along $a$, pushing all arcs which intersect $a$ along as in Figure~\ref{fig: push}.

\begin{figure}[htbp]
    \centering
    \begin{overpic}{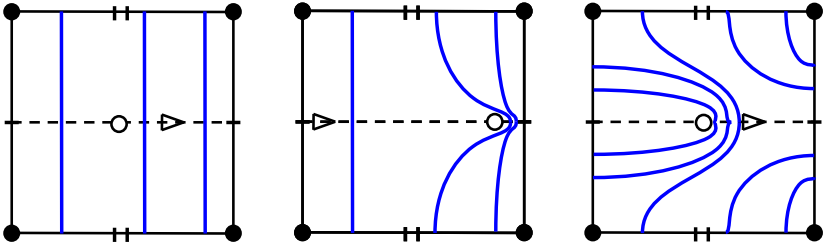}
    \end{overpic}
    \put(-332,69){$a$}
    \caption{The map $P_a$, a ``push" along $a\in\theta$.}
    \label{fig: push}
\end{figure}

\end{defn}

We ensure that Push$(\theta)$ is well-defined for each $\theta\in\pi_1(\oncepuncturedtorusblack,\circ)$ by considering two maps $P_a$ and $P_{a'}$ for $a,a'\in\theta$ and extending the homotopy between $a$ and $a'$ to an isotopy of $\twicepuncturedtorus$. Although the exact way to do this is not obvious, the procedure is fully described in Sections 4.2.1 - 4.2.3 of \cite{MC}. From here on, we discuss the isotopy classes, Push$(\theta)$, in terms of their behavior on tri-arcs, since each $P_a \in$ Push$(\theta)$ has the same behavior on intersecting arcs.

\begin{remark}\label{rmk: push morphism}
The map Push$: \pi_1(\oncepuncturedtorusblack,\circ)\to \pmodtorus$ is a group homomorphism by construction. Thus, for any $\theta_1,\theta_2 \in\pi_1(\oncepuncturedtorusblack,\circ)$, if $\theta=\theta_1\theta_2$, then Push$(\theta)$ pushes along $\theta_1$ then $\theta_2$, which is the same as applying Push$(\theta_2)\circ$ Push$(\theta_1)$.
\end{remark}

The fact that Push acts as a group morphism allows us to prove explicitly that any push map acts on a tri-arc purely by big flips, giving us the precise maps we have been looking for.

Before stating the following Lemma, we must define what we mean by $f(T_*)$ for a homeomorphism $f$ and a tri-arc $T_*$. If we let $T_*=\{a,b,c\}$, then $a,b,c$ are equivalence classes in $\A$ denoted by $a_{\A},b_{\A},c_{\A}$. For a function $f$, we define $f(T_*):=\{[f(a)]_{\A},[f(b)]_{\A},[f(c)]_{\A}\}$.

\begin{lem}\label{lem: Push, big flips}
For any tri-arc $T_*$ and $\theta\in\pi_1(\oncepuncturedtorusblack,\circ)$, the resulting tri-arc $T_*':=P_a(T_*)$ differs from $T_*$ by an even number of big flips for any $P_a \in \Push(\theta)$.
\end{lem}

\begin{proof}
Let $T_*=\{a_{\A},b_{\A},c_{\A}\}$ with corresponding representatives $\hat{a},\hat{b},\hat{c}$. Choose orientations to obtain an oriented tri-arc $T_{\pi_1}=\{a,b,c\}$. Denote $\alpha:=a,\beta:=b\in\pi_1(T^2\setminus\{\circ\},\bullet)$ for simplicity. Note that $\pi_1(T^2\setminus\{\circ\},\bullet)\cong\pi_1(\oncepuncturedtorusblack,\circ)$, so we may find an isomorphism $\phi:\pi_1(T^2\setminus\{\circ\},\bullet)\to\pi_1(\oncepuncturedtorusblack,\circ)$ and let $\tilde{\alpha}:=\phi(\alpha),\tilde{\beta}:= \phi(\beta)$ be generators of $\pi_1(\oncepuncturedtorusblack,\circ)\cong\Z\ast\Z$ (Remark~\ref{rmk: fund grp generators} allows us to do this). Thus $\theta=\tilde{\alpha}^{n_1}\tilde{\beta}^{m_1}\cdots \tilde{\alpha}^{n_k}\tilde{\beta}^{m_k}$ for $(m_i)_{i=1}^k, (n_i)_{i=1}^k\subset\Z$. By Remark~\ref{rmk: push morphism}, we may assume that Push$(\theta)$ is the composition of finitely many maps of the form Push$(\psi)$, for $\psi\in\{\tilde{\alpha},\tilde{\beta},\tilde{\alpha}^{-1},\tilde{\beta}^{-1}\}$. We may cut along $\hat{a}$ and $\hat{b}$ in a way so that $\hat{c}$ is in our desired position of the obtained square representation. Without loss of generality, assume that $c=ab$ and thus our image is as in Figure~\ref{fig: oriented triarc}. We now claim that each push map $P\in$ Push$(\psi)$ will perform a sequence of two big flips on $T_*$. Note that, since $\psi\in\{\tilde{\alpha},\tilde{\beta},\tilde{\alpha}^{-1},\tilde{\beta}^{-1}\}$, the map $P$ will be isotopic in $\homeoplus$ to a push map $\tilde{P}\in\{P_{\hat{a}},P_{\hat{b}},P_{\hat{a}^{-1}},P_{\hat{b}^{-1}}\}$. Restricting this isotopy to the curves in the resulting oriented tri-arc $P(\hat{a}), P(\hat{b}), P(\hat{c})$ shows that $P(T_{\pi_1}) = \tilde{P}(T_{\pi_1})$ since the isotopy classes are the same. This allows us to check only the cases $P\in \{P_{\hat{a}},P_{\hat{b}},P_{\hat{a}^{-1}},P_{\hat{b}^{-1}}\}$. Figure~\ref{fig: twopush} shows the outcomes of each of these push maps. Using our formula for big flip proven in Lemma~\ref{conj: cyclicperm}, we may check each case for $\psi$. First, consider the case $\psi=\tilde{\alpha}$. We first perform a big flip on $b=a^{-1}c$, giving $b'=[c,a^{-1}]b=ca^{-1}c^{-1}ab=aba^{-1}$. We then perform a big flip on $c=b'a$ to get $c'=[a,b']c=aaba^{-1}a^{-1}ab^{-1}a^{-1}c=a^2ba^{-1}$, with the resulting tri-arc $[a,b',c']$. When $\psi = \tilde{\beta}$, we perform a big flip first on $c$ and then on $a$. Since $c = ab$, $c'=[b,a]a=bab^{-1}a^{-1}ab = ba$. Then, $a = b^{-1}c'$, so $a'=[c',b^{-1}]a=bab^{-1}a^{-1}b^{-1}ba=bab^{-1}$. Next, suppose $\psi = \tilde{\alpha^{-1}}$. We now perform a big flip on $c$, followed by $b$. We saw in the above case that $c'=ba$, so $b = c'a^{-1}$ and thus $b'=[a^{-1},c']b=a^{-1}baaa^{-1}b^{-1}b=a^{-1}ba$. Lastly, assume $\psi=\tilde{\beta}$. In this case, we perform a big flip on $a$ and then $c$. Since $a=cb^{-1}$, $a'=[b^{-1},c]a=b^{-1}abbb^{-1}a^{-1}a=b^{-1}ab$. Then, $c=ba'$, so $c'=[a',b]c=b^{-1}abbb^{-1}a^{-1}bb^{-1}ab=b^{-1}ab^2$. All cases are the same as in the images of Figure~\ref{fig: twopush}.
\end{proof}

\begin{figure}[htbp]
    \centering
    \begin{overpic}
    [scale=.7]{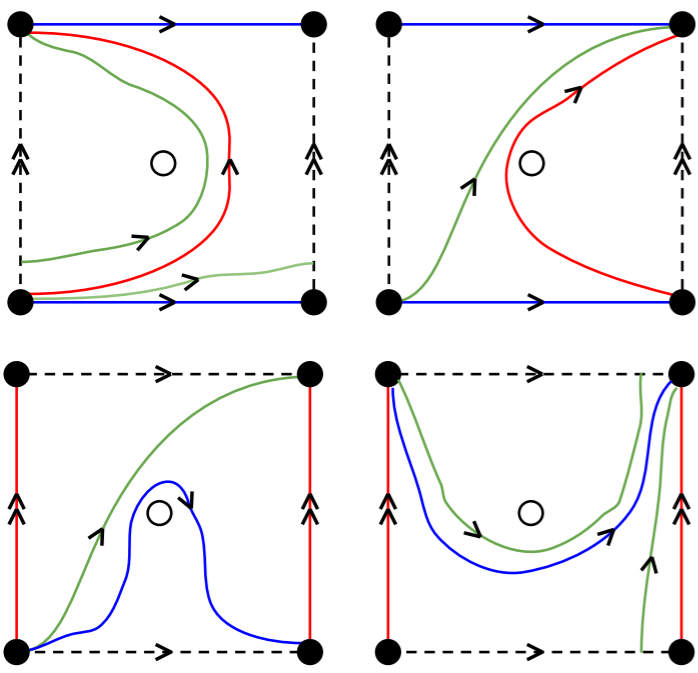}
    \end{overpic}
    \put(-220,115){Push along $\hat{a}$}
    \put(-92,115){Push along $\hat{a}^{-1}$}
    \put(-220,-10){Push along $\hat{b}$}
    \put(-92,-10){Push along $\hat{b}^{-1}$}
    \caption{The images of $\hat{a}$ (blue), $\hat{b}$ (red), and $\hat{c}$ (green).}
    \label{fig: twopush}
\end{figure}
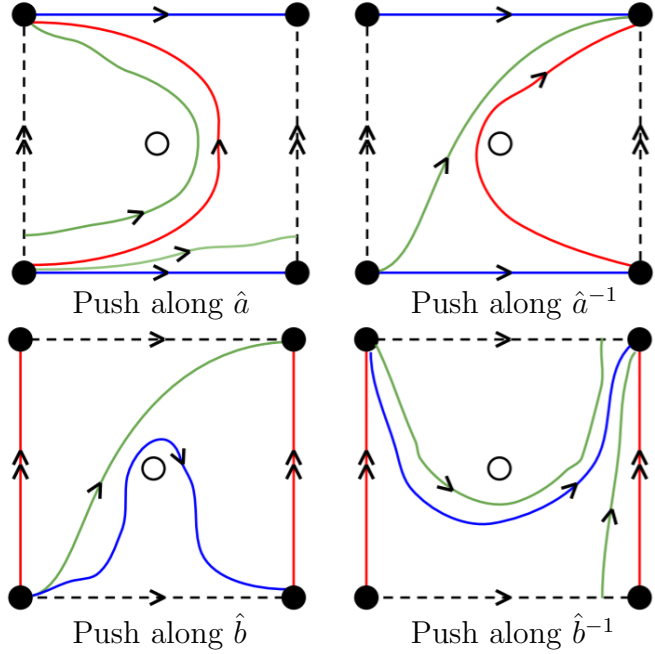

The proof of this lemma also shows us that, given a tri-arc $T_*$, we may choose the class $\theta$ so that Push$(\theta)$ performs a big flip on excactly two homotopy classes in $T_*$. Because push maps have such nice behavior, we want to investigate further ways to classify push maps. Because Push is a morphism,
$$\Push(\theta^{-1})\circ\Push(\theta) = [id]_{\pmodtorus} $$
for all $\theta \in \pi_1(\oncepuncturedtorusblack,\circ)$. That is, push maps $P_a$ are undone by simply pushing along $a$ backwards. This makes us want to determine if push maps will all be isotopic to the identity in $\homeoplus$. However, we are not able to construct one in this fashion (by just pushing along $a^{-1}$), since all intermediate maps would need to fix $\circ$. Luckily, if we consider these homeomorphisms in the space Homeo$^+(\oncepuncturedtorusblack,\{\bullet\})$, we no longer need to fix $\circ$, allowing us to create an isotopy to the identity. To more formally describe this property, we are motivated to define a ``Forget" map.

\begin{defn}\label{defn: forget}
We define the map Forget$:\pmodtorus\to$ PMod$(\oncepuncturedtorusblack)$ as follows: let $[f]_{\pmodtorus}\in\pmodtorus$ and consider the representative $f$ as a map in Homeo$^+(\oncepuncturedtorusblack,\{\bullet\})$ with $f(\circ)=(\circ)$. Then Forget$([f]_{\pmodtorus}) = [f]_{\text{PMod}(\oncepuncturedtorusblack)}$, the isotopy class of $f$ in $\oncepuncturedtorusblack$. That is, this map ``forgets" the puncture $\circ$.
\end{defn}

With the process described above, we can see that Forget(Push$(\theta)$) should always be the identity element. If we take $P_a$, for $a\in\theta$, then this map will be isotopic to the identity in Homeo$^+(\oncepuncturedtorusblack,\{\bullet\})$ since $\circ$ need not be fixed by the intermediate homeomorphisms. Then, we can travel along $a^{-1}$, which will bring the images of all curves intersecting $a$ back to their pre-images, yielding the identity map. This implies that Image(Push) $\subseteq$ Ker(Forget). However, we wish to determine if push maps describe all homeomorphisms which are isotopic to the identity in Homeo$^+(\oncepuncturedtorusblack)$, a result that is not at all obvious to show. Luckily, \cite{MC} discusses construction of the Birman exact sequence,
$$1\to\pi_1(\oncepuncturedtorusblack,\circ)\overset{Push}{\longrightarrow}\pmodtorus\overset{Forget}{\longrightarrow}\text{PMod}(\oncepuncturedtorusblack)\to 1. $$
As proven in Section 4.2.3 of \cite{MC}, this sequence is exact (that is, Image(Push) = Ker(Forget). Within this source, we can also find a proof of our desired result. 

\begin{lem}\label{lem: push image}
Image(Push) $=$ Ker(Forget).
\end{lem}

Thus, push maps do describe all maps in $\pmodtorus$ which are isotopic to the identity in Homeo$^+(\oncepuncturedtorusblack,\{\bullet\})$. While this result is not only a nice characterization of push maps, it also gives us another way to locate then. Since push maps are difficult to explicitly construct, it is useful to study tools which allow us to ``find" push maps within $\pmodtorus$. Lemma~\ref{lem: fund grp identity} gives us one such tool. Note that the proof of the lemma's statement also depends on a result from \cite{MC}. There is an isomorphism $\sigma:\text{PMod}(\oncepuncturedtorusblack)\to\text{SL}(2,\Z)$ defined as follows: 

First, take any $[f]\in\text{PMod}(\oncepuncturedtorusblack)$ and a representative $f$. By definition of PMod$(\oncepuncturedtorusblack)$, $f$ is an isomorphism of $T^2$ and thus induces an isomorphism $f_*:\Z\times\Z\to\Z\times\Z$. Note that group of isomorphisms on $\Z\times\Z$ is itself isomorphic to SL$(2,\Z)$ by taking an isomorphism $g$ and constructing a matrix $\begin{pmatrix}
a & b \\
c & d
\end{pmatrix}$ with $g(1,0)=(a,c)$ and $g(0,1)=(b,d)$. We apply this process to $f_*$ in order to find the matrix $\sigma[f]$. The proof that this map is an isomorphism is discussed in Section 2.2.4 of \cite{MC}.

\begin{lem}\label{lem: fund grp identity}
If a map $\phi\in\homeoplus$ has induced group morphism $\phi_*:\pi_1(T^2,\bullet)\to\pi_1(T^2,\bullet)$ that is the identity, then $[\phi]_{\text{PMod}(\oncepuncturedtorusblack)} =$ Push$(\theta)$ for some $\theta\in\pi_1(\oncepuncturedtorusblack,\circ)$.
\end{lem}

\begin{proof}
Note that, by definition of $\homeoplus$, $\phi$ is a homeomorphism of $T^2$ which fixes $\bullet$, and is thus an element of Homeo$^+(\oncepuncturedtorusblack,\{\bullet\})$. By appealing to the construction of $\sigma$ defined above, $\sigma([\phi]_{\text{PMod}(\oncepuncturedtorusblack)}) = \begin{pmatrix}
a & b \\
c & d
\end{pmatrix}$, where $(a,c) = \phi_*(1,0)$ and $(b,d)=\phi_*(0,1)$. However, since $\phi_*$ is the identity, $\sigma([\phi]_{\text{PMod}(\oncepuncturedtorusblack)})=\begin{pmatrix}
1 & 0 \\
0 & 1
\end{pmatrix}$. Because $\sigma$ is group isomorphism which maps $[\phi]_{\text{PMod}(\oncepuncturedtorusblack)}$ to the identity element, $[\phi]_{\text{PMod}(\oncepuncturedtorusblack)}$ must be the identity element of $\text{PMod}(\oncepuncturedtorusblack)$. Note now that by definition of the Forget map, $$[\phi]_{\text{PMod}(\oncepuncturedtorusblack)}=\text{Forget}([\phi]_{\pmodtorus}),$$ which implies that
$$[\phi]_{\pmodtorus} \in \text{Ker(Forget)} = \text{Image(Push)}.$$
That is, there exists a $\theta\in\pi_1(\oncepuncturedtorusblack,\circ)$ so that $[\phi]_{\pmodtorus}=$ Push$(\theta)$.
\end{proof}

We then reach yet another description of push maps: the elements of $\pmodtorus$ whose representatives all induce the identity morphism on $\pi_1(T^2,\bullet)$. However, it is still not easy to find many homeomorphisms of the torus which act as identities on homotopy classes. In fact, we are not used to defining homeomorphisms based on their effect on curves at all. For this, we appeal again to \cite{MC} in order to define homeomorphisms based on mappings of curves, giving us the following result.

\begin{lem}[Change of Coordinates Principle for Pairs Intersecting Once]\label{lem: change of coords}
Suppose that $a,b$ and $a',b'$ are pairs of simple closed curves in surface $S$ which intersect exactly once. There exists an orientation-preserving homeomorphism $\phi:S\to S$ such that $\phi(a)=a'$ and $\phi(b)=b'$.
\end{lem}

The proof of this statement is discussed in Sections 1.3.1-1.3.2 of \cite{MC}. The change of coordinates principle gives us a nice way to find homeomorphisms in $\homeoplus$ that act as we wish on curves.

Note that for this proof, we will be considering both homotopy classes of unoriented arcs in $\A$ and homotopy classes of oriented arcs in the fundamental group. For the sake of notation, we will notate homotopy classes in $\A$ by $a_{\A}$, with homotopy classes in the fundamental group being notated simply by $a$. Further, it will make things easier if we define a new set
$$\A'=\faktor{\{a:[0,1]\to T^2: a\text{ unoriented essential simple arc with } a(0)=a(1)=\bullet\}}{\simeq}$$
which only differs from $\A$ by considering homotopy of arcs in $T^2$ instead of $T^2\setminus\{\circ\}$. We will denote elements of $\A'$ by $a_{\A'}$. For $a\in\A$, we will be considering the homotopy class $[i(\hat{a})]_{\A'}$, where $i$ is the inclusion map $\oncepuncturedtorus\hookrightarrow T^2$ and $\hat{a}$ is a representative of $a$. We will let $i(a)_{\A'}$ denote this homotopy class, where $i(a)$ denotes an oriented homotopy class in $\pi_1(T^2,\bullet)$. However, for a representative $\hat{a}\in a$, $i(\hat{a})\subset T^2$ still denotes the image of $\hat{a}$ in $T^2$.

\begin{manualtheorem}{\ref{thm: fibers connected}}
For any $v \in \fareydual$, the fiber $\pi^{-1}(v)$ is connected.
\end{manualtheorem}

\begin{proof}
Let $v\in \fareydual$, and assume $T,T'\in\tripantsgraph$ such that $T,T' \in \pi^{-1}(v)$. Let $T_*=\{a_{\A},b_{\A},c_{\A}\}$, $T_*'=\{a'_{\A},b'_{\A},c'_{\A}\}$ be the corresponding tri-arcs. We will prove that there is a finite sequence of big flips separating $T_*$ and $T_*'$, implying the existence of a sequence of edges in $\pi^{-1}(v)$ connecting the vertex $T_*$ to $T_*'$. Choose representatives $\hat{a},\hat{b},\hat{c},\hat{a}',\hat{b}',\hat{c}'$ that realize the tri-arcs $T_*$ and $T_*'$.  Note first that the hypothesis implies that
$$i(a)_{\A'}=i(a')_{\A'},\:\:i(b)_{\A'}=i(b')_{\A'},\;\;i(c)_{\A'}=i(c')_{\A'}, $$
up to relabeling. Applying Lemma~\ref{lem: change of coords} for $S=T^2\setminus\{\circ\}$, we get a homeomorphism $\phi\in\homeoplus$ such that $\phi(\hat{a})=\hat{a}'$ and $\phi(\hat{b})=\hat{b}'$. Note that $\phi:T^2\to T^2$ fixes $\circ$, so $\psi:=\phi|_{\oncepuncturedtorus}$ is an orientation-preserving homeomorphism which fixes $\bullet$ (since it sends the intersection point of $\hat{a},\hat{b}$ to the intersection point of $\hat{a}',\hat{b}'$) and has the property $\phi\circ i\equiv i\circ\psi$. Thus, $\psi(\hat{a})=\phi(\hat{a})=\hat{a}'$ and $\psi(\hat{b})=\phi(\hat{b})=\hat{b}'$. Further, these maps generate isomorphisms $\phi_*:\pi_1(T^2,\bullet)\to\pi_1(T^2,\bullet)$ and $\psi_*:\pi_1(\oncepuncturedtorus,\bullet)\to\pi_1(\oncepuncturedtorus,\bullet)$. In order to study these isomorphisms, we must choose orientations to obtain an orientated tri-arc $T_{\pi_1}=\{a,b,c\}\subset\pi_1(\oncepuncturedtorus,\bullet)$, which gives us natural orientations for the inclusions of these classes $i(a),i(b),i(c)\in\pi_1(T^2,\bullet)$. Note that we can define orientations for $a',b',c'$ (and thus $i(a'),i(b'),i(c')$) by composing parameterizations of these oriented curves with $\psi$. It follows that $i(a)=i(a'),i(b)=i(b')$, and $i(c)=i(c')$ as elements of $\pi_1(T^2,\bullet)$. We first study $\phi_*$ and choose a generating set $i(a), i(b)$ of $\pi_1(T^2,\bullet)\cong\Z\times\Z$. Because $$\phi(i(a))=[\phi(i(\hat{a}))]=i(a')\text{  and  }\phi(i(b))=[\phi(i(\hat{b}))]=i(b'),$$ we see that $\phi_*$ fixes a generating set and is thus the identity map on $\pi_1(T^2,\bullet)$. Applying Lemma \ref{lem: fund grp identity} and then Lemma \ref{lem: Push, big flips} to $\phi$  tells us that $$\phi(T_*)=\{a'_{\A},b'_{\A},\phi(c)_{\A}\}=\{a'_{\A},b'_{\A},\psi(c)_{\A}\}=\psi(T_*)$$ differs from $T_*$ be an even number of big flips. Note that $\psi(T_*)$ is necessarily a tri-arc since it is the image of a tri-arc by a homeomorphism. Note now that $\phi_*$ being the identity implies that $\phi(i(c))_{\A'}=i(c')_{\A'}$. By assumption $\phi(i(\hat{c}))=i(\psi(\hat{c}))$, so we then have $$i(\psi(c))_{\A'}=i(c')_{\A'}.$$ Note that since $\psi(T_*)$ is a tri-arc, this equation implies that $\pi(\psi(T_*)) = \pi(T_*')$. Moreover, if $\psi(c)_{\A}\neq c'_{\A}$ then $\psi(c)_{\A}$ differs from $c'_{\A}$ by either a small or a big flip. By Definition~\ref{defn: pitilde}, a small flip is sent by $\pi$ onto an edge of $\mathcal{F}^*$ connecting distinct vertices, so the former case implies that $\pi(\psi(T_*))\neq \pi(T_*')$, contradiction. Thus either $\psi(c)_{\A} = c'_{\A}$ or $\psi(c)_{\A}$ is a big flip of $c'_{\A}$. In either case, $T_*$ differs by $T_*'$ by a finite sequence of big flips.
\end{proof}

With connectedness of the fibers, we have tackled the first part of the proof. We are now able to prove Theorem~\ref{thm: tri-pants graph connected} by appealing to connectedness of $\fareydual$.

\begin{manualtheorem}{\ref{thm: tri-pants graph connected}}
The tri-pants graph $\tripantsgraph$ is connected.
\end{manualtheorem}

\begin{proof}
Let $T,T'\in\tripantsgraph$ be two tri-pants. We will prove that there is a sequence of edges in $\tripantsgraph$ which connects $T$ to $T'$. Let $v=\pi(T), v' = \pi(T')\in\fareydual$ be two vertices. Since $\fareydual$ is connected, there is a sequence of vertices $v_0,v_1,\dots,v_k,v_{k+1}$ (where $v=v_0, v_{k+1}=v'$) and a sequence of edges $e_0,e_1,\dots,e_k$ both contained in $\fareydual$ and so that $e_i$ connects $v_i$ to $v_{i+1}$ for all $i\in\{0,1,\dots,k\}$. Because $\pi$ is surjective by Lemma~\ref{lem: pisurj}, for each edge $e_i$ in $\tripantsgraph$, there is at least one corresponding edge $w_i \in \pi^{-1}(e_i)\subset\tripantsgraph$ which connects some tri-pant $T_i\in\pi^{-1}(v_i)$ to another $T_{i+1}\in\pi^{-1}(v_{i+1})$. Starting with $T\in\pi^{-1}(v_0)$, we find an edge $w_0\in\pi^{-1}(e_0)$ which connects some $T_0\in\pi^{-1}(v_0)$ to a $T_1\in\pi^{-1}(v_1)$. By Theorem~\ref{thm: fibers connected}, there is a sequence of edges in the fiber $\pi^{-1}(v_0)$ which connects $T$ to $T_0$. Adding the edge $w_0$ at the end of this sequence gives a path from $T$ to $T_1$. Continuing with this pattern, after adding the edge $w_k$ we have found a path in $\tripantsgraph$ from $T$ to some $T_{k+1}\in\pi^{-1}(v')$. Because $T' \in \pi^{-1}(v')$ as well, Theorem \ref{thm: fibers connected} implies the existence of a path within the fiber $\pi^{-1}(v')$ connecting $T_{k+1}$ to $T'$. Therefore, there is a path in $\tripantsgraph$ connecting $T$ to $T'$.
\end{proof}

\subsubsection{Infinite Diameter}
We will again use the connection between the tri-pants graph and the Farey graph to prove Theorem~\ref{thm: infinitediameter}. 
\begin{manualtheorem}{\ref{thm: infinitediameter}}
The tri-pants graph $\tripantsgraph$ has infinite diameter.
\end{manualtheorem}
\begin{proof}
Recall that $\mathcal{F}^*$ is a connected tree with infinite diameter. Thus, for any $n\in\mathbb{Z}_{\ge0}$ there exists vertices $v_1, v_2\in \mathcal{F}^*$ such that $d_{\mathcal{F}^*}(v_1, v_2)\geq n$. Let $T\in\tilde{\pi}^{-1}(v_1)$ and $T'\in\tilde{\pi}^{-1}(v_2)$. We claim that any path $\gamma$ connecting $T$ and $T'$ will have length at least $n$. For the sake of contradiction, assume that there exists some path $\delta$ between $T$ and $T'$ with length less then $n$. Then $\pi(\delta)$ provides a map between $v_1$ and $v_2$ with length less than or equal to $n$ since $\pi$ sends edges in $\tripantsgraph$ to either vertices or edges in $\fareydual$. If this length is less than $n$, we contradict the statement that $\mathcal{F}^*$ has infinite diameter. Thus, for any $n\in\mathbb{Z}_{\ge0}$ there exists $T,T'\in\tripantsgraph$ with distance at least $n$, hence the tri-pants graph has infinite diameter. 
\end{proof}

\section*{Conclusion}
Throughout this paper, we have investigated the properties of homotopy classes of essential simple closed curves on $\twicepuncturedtorus$ which form pants decompositions, with a large focus on how different decompositions can relate to one another. In this study, the object of a tri-pant arose when considering triples of pants decompositions which pairwise intersect once. After presenting the primary definition of a tri-pant, Definition~\ref{def: tripantsI}, we were able to prove equivalence with a second definition: a maximal collection of homotopy classes of essential simple closed curves on $\twicepuncturedtorus$ which pairwise intersect at most once. Since tri-pants still proved difficult to picture on the twice-punctured torus, we found it easier to relate these structures to collections of homotopy classes of essential simple arcs on the once-punctured torus which intersect only at the base point and do not pairwise bound cylinders. Luckily for us, we were able to prove a one-to-one correspondence between tri-pants and these new objects, tri-arcs. Because we were able to view tri-arcs on the punctured square representation of $\oncepuncturedtorus$, this correspondence gave a great pictorial view of tri-pants. 

In fact, to our delight, we could then also relate these new structures to the Farey graph $\mathcal{F}$ by utilizing the map $i_*:\pi_1(\oncepuncturedtorus,\bullet)\to\pi_1(T^2,\bullet)$. By mapping the three representative arcs of a tri-arc to three corresponding simple closed curves on $T^2$, we were able to relate a tri-pant to a triangle on the Farey graph, or a vertex on its dual $\fareydual$. Mapping between sets of tri-pants and the Farey graph motivated us to create a graph whose vertices are tri-pants. However, as it was unclear when vertices of two tri-pants should be connected by an edge, we introduced elementary moves. Because two adjacent triangles in $\mathcal{F}$ correspond to triples which shared two of the same curves, it made sense for two tri-pants to be connected by an edge if they shared two pants decompositions. In the process of defining elementary moves, we also introduced big and small flips to explain elementary moves between two tri-pants as viewed on their associated tri-arcs. With this understanding, we could then define the tri-pants graph $\tripantsgraph$.

In our study of $\tripantsgraph$, we were able to formally define their association with $\fareydual$ by introducing the map $\pi:\tripantsgraph\to\fareydual$. Although $\pi$ is not a bijection, it is a surjective map. Further, studying $\pi$ allowed us to prove, among other results, that every vertex in the graph has valence equal to $9$ and that $\tripantsgraph$ is connected, is not a tree, and has infinite diameter. We hope that this paper provides a robust exposition to the tri-pants graph, although many questions pertaining to $\tripantsgraph$ remain unanswered. In particular, we conjecture that
each fiber of the map $\pi$ is isomorphic to the dual of the Farey graph $\fareydual$. There is a fundamental similarity between big flips and diagonal exchanges on triangulations of $\oncepuncturedtorus$, which we believe may prove that each fiber of the map $\pi$ is infinite.

Finally, we hope to inspire similar work on other ``small" surfaces. For instance, does such a graph structure exist atop the thrice-punctured torus, and if so, is there a relationship between this construction and the tri-pants graph? This area of study is relatively unexplored, and we eagerly await future findings.
\emergencystretch=1em

\bibliographystyle{hamsplain}
\bibliography{biblio}
\end{document}